\documentclass[a4paper,11pt,twoside,notitlepage,reqno]{amsart}
\usepackage{amsmath,amssymb,amsbsy,amsthm}
\usepackage[hmargin=1.4in,vmargin=1.4in]{geometry}
\usepackage{url}

\usepackage{tikz}
\usepackage[linktocpage]{hyperref}
\hypersetup{
    colorlinks=true,        % false: boxed links; true: colored links
    linkcolor=red,          % color of internal links (change box color with linkbordercolor)
    citecolor=red,         % color of links to bibliography
    filecolor=magenta,      % color of file links
    urlcolor=cyan           % color of external links
}

\addtocontents{toc}{\protect\setcounter{tocdepth}{1}}

\usepackage{eucal}

\theoremstyle{plain}
	\newtheorem{thm}{Theorem}
	
		\numberwithin{thm}{section}
	\newtheorem{lemma}[thm]{Lemma}
	\newtheorem{prop}[thm]{Proposition}
	\newtheorem{cor}[thm]{Corollary}

	\newtheorem*{thm*}{Theorem}
	\newtheorem*{lemma*}{Lemma}
	\newtheorem*{prop*}{Proposition}
	\newtheorem*{cor*}{Corollary}
	\newtheorem*{conj*}{Conjecture}

\theoremstyle{definition}
	\newtheorem{example}[thm]{Example}
	\newtheorem*{example*}{Example}
	
	\newtheorem{question}{Question}

	\newtheorem{remark}[thm]{Remark}
\begin{document}

%\title{Rational orbits with values in a group of $S$-units}

\title[Affine Representability and decision procedures]{Affine representability and decision procedures for commutativity theorems for rings and algebras}

\author[J. P. Bell]{Jason P. Bell}
\address{Department of Pure Mathematics, University of Waterloo, Waterloo, ON Canada N2L 3G1}
\email{jpbell@uwaterloo.ca}

\author[P. V. Danchev]{Peter V. Danchev}
\address{Institute of Mathematics and Informatics, Bulgarian Academy of Sciences, ``Acad. G. Bonchev" str., bl. 8, 1113 Sofia, Bulgaria}
\email{danchev@math.bas.bg; pvdanchev@yahoo.com}


\begin{abstract} We consider applications of a finitary version of the Affine Representability theorem, which follows from recent work of Belov-Kanel, Rowen, and Vishne. Using this result we are able to show that when given a finite set of polynomial identities, there is an algorithm that terminates after a finite number of steps which decides whether these identities force a ring to be commutative.  We then revisit old commutativity theorems of Jacobson and Herstein in light of this algorithm and obtain general results in this vein.  In addition, we completely characterize the homogeneous multilinear identities that imply the commutativity of a ring.
\end{abstract}

\maketitle
\section{Introduction and Background}

One of the crowning achievements in the theory of rings satisfying a polynomial identity is Kemer's solution in characteristic zero \cite{K1}--\cite{K4} to the Specht problem \cite{S}, which asks whether every set of polynomial identities of an algebra is a consequence of a finite number of identities.  A key component of Kemer's work is the Affine Representability theorem (see \cite{Alja}), which shows that for a finitely generated algebra $A$ satisfying an identity over an infinite field, there is a finite-dimensional algebra $B$ (possibly over a larger field) that satisfies the exact same set of identities as $A$. The original groundbreaking work of Kemer has since been expanded by others, including notably Belov-Kanel, Rowen, and Vishne \cite{B1,B2,B3,BRV1,BRV2,BRV3}, who extended much of Kemer's theory to the setting of algebras over commutative noetherian base rings.  It is clear that the Affine Representability theorem does not in general have a positive solution for algebras over a finite field; for example, if $A=\mathbb{F}_p[t]$ then $A$ cannot satisfy the exact same set of identities as a finite-dimensional $\mathbb{F}_p$-algebra $B$, because a finite-dimensional $\mathbb{F}_p$-algebra satisfies an identity of the form $X^m-X^n=0$ for some $m$ and $n$ with $m>n$, while $A$ satisfies no such identity.  In practice, however, one is often only concerned with finite sets of non-identities. When one restricts one's attention to this setting, it then becomes a natural question of whether this finitary version of the Affine Representability theorem holds. 
%\begingroup
%\setcounter{tmp}{\value{question}}% store current value of theorem counter
%\setcounter{question}{0} %assign desired value to theorem counter
%\renewcommand\thequestion{\Alph{question}}% locally redefine the representation of the theorem counter

\begin{question} Given sets  $\mathcal{S}$ and $\mathcal{T}$ of polynomial identities with $\mathcal{T}$ finite, if there is an algebra which satisfies all of the identities from $\mathcal{S}$ and none of the identities from $\mathcal{T}$, is there a finite ring with this property?
\label{A}
\end{question}

%\endgroup
Although this question is not explicitly answered in the literature, recent work of Belov-Kanel, Rowen, and Vishne \cite{BRV1, BRV2, BRV3} can be used to quickly show that this question has an affirmative answer (see Theorem {\ref{lem:p} for details).

Historically, the most important results involving both polynomial identities and polynomial non-identities have largely focused on the case when a collection of identities force $[X,Y]=0$ to also be an identity; that is, to show that an identity or family of 
identities force a ring to be commutative (see Herstein \cite[Chapt. 6]{Her} and the references given at the end of the chapter along with the paper of Pinter-Lucke \cite{PL}). This, for example, is the content of special cases 
of famous results of Jacobson \cite{J} and Herstein \cite{H}, which respectively assert that rings for which $X^n=X$ and $[X^n-X,Y]=0$ are identities for some fixed $n\ge 2$ are necessarily commutative.\footnote{The general version of Jacobson's result says that if for each $x$ in a ring $R$ there is some $n=n(x)\ge 2$ such that $x^n=x$ then $R$ is commutative and Herstein's result is similar, but instead now just requiring that $x^n-x$ be central in $R$.} Results of this type are 
typically called \emph{commutativity theorems}. 

An affirmative answer to Question \ref{A} gives a general approach to attacking such problems, as it reduces the analysis to looking at finite rings, although we are unaware of this approach being 
used previously.  In fact, in the special case where one is looking at $[X,Y]=0$ being a non-identity for a ring, one can give a more direct answer to Question \ref{A} that does not require the deep machinery coming 
out of the recent work of Belov-Kanel, Rowen, and Vishne on Kemer's theorem (see Theorem \ref{thm:Specht}).

One of the consequences of the fact that Question \ref{A} has an affirmative answer is that if there is a noncommutative ring that satisfies some set of identities, then there is a finite noncommutative ring satisfying all identities in the set. We use this observation to give a decision procedure to determine whether a finite set of polynomial identities forces a ring to be commutative.  In light of this result, we are able to prove the following somewhat unexpected result, which can be viewed as a sort of machine for producing commutativity theorems.
\begin{thm} There is a decision procedure that takes as input a finite number of polynomial identities $P_1=\cdots =P_m=0$ and gives as output either a finite noncommutative ring for which these identities all hold or shows that every ring for which these identities all simultaneously hold is commutative.
\label{thm:algorithm}
\end{thm}
By a \emph{decision procedure} we simply mean there is an algorithm that terminates after a finite number of steps.  The algorithm itself is rather lengthy to describe in general, but in practice---for specific sets of identities---it can be done reasonably quickly and we give applications of our algorithm in Section \ref{Applications}.  To produce an algorithm, we require a coarse classification of finite noncommutative rings with the property that all non-trivial homomorphic images and all proper subrings are commutative.  This is the content of Theorem \ref{thm:trichotomy} and Remark \ref{rem:trichotomy}; Theorem \ref{thm:trichotomy} is somewhat technical, but it shows that all such algebras lie in one of three infinite classes of algebras that are indexed by the prime numbers.

As a quick application of our Theorem \ref{thm:trichotomy}, we are able to completely characterize the homogeneous multilinear polynomial identities that force a ring to be commutative.  Arguably the most important class of polynomial identities are the homogeneous multilinear identities, which arise in the theory of polynomial identities via a natural linearization process.  To give this characterization, we fix a homogeneous multilinear polynomial
$$P(X_1,\ldots ,X_m)=\sum_{\sigma \in S_m} c_{\sigma}X_{\sigma(1)}\cdots X_{\sigma(m)} \in \mathbb{Z}\{X_1,\ldots ,X_m\}.$$
%For each $i\in \{1,\ldots ,m\}$ we define commutative polynomials
%$$\Omega_i(P):=\sum_{k=1}^m \sum_{\{\sigma\in S_m \colon \sigma(k)=i\}} c_{\sigma} \prod_{j<k} Y_{\sigma(j)} \cdot \prod_{\ell>k} Z_{\sigma(\ell)} \in \mathbb{Z}[Y_1,\ldots ,Y_m,Z_1,\ldots ,Z_m].$$
For each $i,j\in \{1,\ldots ,m\}$ with $i<j$ we define
\begin{equation}
\Theta_{i,j}(P) = \sum_{\{\sigma\in S_m\colon \sigma^{-1}(i)<\sigma^{-1}(j)\}} c_{\sigma}\in \mathbb{Z}.
\label{eq:Theta}
\end{equation}
Then we have the following result.

\begin{thm}
Let $P(X_1,\ldots ,X_m)=\sum_{\sigma \in S_m} c_{\sigma}X_{\sigma(1)}\cdots X_{\sigma(m)} \in \mathbb{Z}\{X_1,\ldots ,X_m\}$ be a homogeneous multilinear polynomial.  Then there is a noncommutative ring for which $P=0$ is an identity if and only if there is a prime number $p$ such that the following hold:
\begin{enumerate}
\item $p\mid P(1,1,\ldots ,1)$;
\item $p\mid \Theta_{i,j}(P)$ for $1\le i<j\le m$. 
\end{enumerate}
 Moreover, if there is such a prime $p$ for which these conditions hold, then $P=0$ is an identity for the noncommutative ring $\mathbb{F}_p\{U,V\}/(U^2,V^2,UV)$.  
 \label{thm:multi}
\end{thm}
We have chosen to state the conditions in Theorem \ref{thm:multi} in terms of $P(1,\ldots ,1)$ and the $\Theta_{i,j}$, as these are integers that one can explicitly compute and so one can determine whether a prime $p$ holds for which (1) and (2) hold.  Nevertheless, for a homogeneous multilinear polynomial $P(X_1,\ldots, X_s)$, it is more natural to let $P_{i,j}$ denote specialization $P(1,\ldots ,1, X_i, 1,\ldots ,1,X_j,1,\ldots ,1)$ for $i<j$, where we have an $X_i$ in the $i$-th coordinate and $X_j$ in the $j$-th coordinate.  Then $P_{i,j} = \Theta_{i,j}(P)X_i X_j + (P(1,1,\ldots ,1)-\Theta_{i,j})X_jX_i$, and so conditions (1) and (2) are equivalent to the condition that the polynomial $P_{i,j}$ be divisible by the prime $p$ for $1\le i<j\le m$.  

The outline of this paper is as follows.  In \S\ref{Rowen} we show how one can give an affirmative answer to Question \ref{A} problem using the powerful work of Belov-Kanel, Rowen, and Vishne.  In addition, we give a direct argument in the case where one is studying rings for which $[X,Y]=0$ is a non-identity and give a coarse classification of finite noncommutative rings for which every proper homomorphic image and every proper subring is commutative.  In \S\ref{Algorithm}, we use results from \S\ref{Rowen} to give an algorithm, which proves Theorem \ref{thm:algorithm}.  In \S\ref{Applications} we revisit the fixed-degree versions of old commutativity theorems of Jacobson \cite{J} and Herstein \cite{H} in light of these results and prove general commutativity theorems and in \S\ref{sec:multilinear} we characterize the homogeneous multilinear polynomial identities with the property that whenever a ring satisfies this identity it is necessarily commutative and prove a more general version of Theorem \ref{thm:multi} (see Theorem \ref{thm:multi2}). Throughout this paper, we will take a \emph{noncommutative ring} to be a ring that is \emph{not} commutative and all rings considered are assumed to be associative and possessing an identity element. We refer the reader to \cite{BR}, \cite{D}, and \cite{R1} for background on polynomial identities.  Finally, we will often say that a polynomial $P\in \mathbb{Z}\{X_1,\ldots ,X_s\}$ is either an \emph{identity} or a \emph{non-identity} for a ring $R$.  By this, we simply mean that $P(r_1,\ldots ,r_s)=0$ for all $r_1,\ldots, r_s\in R$ when speaking of $P$ being an identity for $R$; and when speaking of $P$ being a non-identity for $R$ we mean that there exist $r_1,\ldots ,r_s\in R$ such that $P(r_1,\ldots ,r_s)\neq 0$.  

\section*{Acknowledgments}
We are grateful to Lance W. Small and to Louis H. Rowen for many useful comments.  The first-named author expresses his thanks to Luna Xin, who asked a question that led to Theorem 4.3. In addition, we are grateful to the anonymous referee, who made numerous helpful comments, including suggesting the addition of Proposition \ref{prop:gen}.
\section{Finite rings and finite sets of identities}
\label{Rowen}

In this section, we show that Question \ref{A} has an affirmative answer. We once again point out that our work relies heavily on the aforementioned work of  Belov-Kanel, Rowen, and Vishne \cite{BRV3}. We begin with a classical fact from commutative algebra, which, if one borrows terminology from group theory, says that finitely generated $\mathbb{Z}$-algebras are \emph{residually finite} (see \cite{Vara}).  (We note that although the paper \cite{CL} predates the reference \cite{Vara}, the paper of Chew and Lawn \cite{CL} deals with what are now called \emph{just infinite ring} and not residually finite rings in the sense given here.)
\begin{lemma} Let $C$ be a finitely generated commutative $\mathbb{Z}$-algebra. If $x\in C$ is nonzero, then there is some ideal $L$ such that $x\not\in L$ and such that $C/L$ is finite.
\label{lem:Krull}
\end{lemma}
\begin{proof} Let $I$ be the annihilator of $x$.  Then since $x$ is nonzero, there is some maximal ideal $Q$ such that $I$ is contained in $Q$. In particular, $x$ has nonzero image in the localization $C_Q$. By the Krull intersection theorem (see, for instance, \cite[Corollary 5.4]{Eisenbud}), we deduce that $$\bigcap_{n\in \mathbb{N}} Q^n C_Q = (0)$$ and so there is some $n$ such that $x\not\in Q^n C_Q$ and thus $x\not\in Q^n$. By the Nullstellensatz \cite[Theorem 4.19, p. 132]{Eisenbud}, $C/Q$ is a finite field, and so every ideal in the chain
$$Q\supseteq Q^2\supseteq Q^3 \supseteq \cdots $$ is cofinite and so letting $L=Q^n$ gives the result.
\end{proof}
\begin{thm} Let $\mathcal{S}$ and $\mathcal{T}$ be sets of polynomial identities with $\mathcal{T}$ finite. If there exists a ring $R$ such that all elements of $\mathcal{S}$ are identities for $R$ and all elements of $\mathcal{T}$ are non-identities for $R$, then there exists a finite ring $S$ such that all elements of $\mathcal{S}$ are identities for $S$ and all elements of $\mathcal{T}$ are non-identities for $S$. \label{lem:p}
\end{thm}
\begin{proof}
We observe that it suffices to prove the case when $|\mathcal{T}|=1$, since if this holds, then for each non-identity $G=0$ in $\mathcal{T}$ there is a finite ring $A$ for which each element of $\mathcal{S}$ is an identity and for which $G=0$ is a non-identity.  Then the direct product of these finite rings we produce is a finite ring with the desired properties.  Thus we assume henceforth that $\mathcal{T}$ consists of a single identity $G(X_1,\ldots ,X_s)=0$.

By assumption there is some $\mathbb{Z}$-algebra $R$ for which all elements of $\mathcal{S}$ are identities and such that $G=0$ is a non-identity.  Then there is some $s$-generated subalgebra $R_0$ of $R$ that witnesses the fact that $G(X_1,\ldots ,X_s)=0$ is a non-identity.
We let $I$ be the $T$-ideal in $S:=\mathbb{Z}\{X_1,\ldots, X_s\}$ generated by the identities from $\mathcal{S}$.  Then since $\mathbb{Z}$ is noetherian, a result of Belov-Kanel, Rowen, and Vishne \cite[\S7.2]{BRV3}, the algebra $A:=\mathbb{Z}\{X_1,\ldots, X_s\}/I$ is representable and hence there is a commutative $\mathbb{Z}$-algebra $C$ such that $A$ embeds as a subalgebra of the full $n\times n$ matrix ring $M_n(C)$ for some $n\ge 1$. We may replace $A$ by its image in $M_n(C)$, and since this image is isomorphic to $A$, it satisfies all identities in $\mathcal{S}$.  In addition, we may replace $C$ by the finitely generated subalgebra generated by the entries of a finite set of generators for $A$, since all we require is that the map $A\to M_n(C)$ be an embedding.  We therefore assume that $C$ is finitely generated and hence noetherian. Since $I$ is a $T$-ideal and since $G=0$ is a non-identity for $A$, the image of $G(X_1,\ldots ,X_s)$ in $M_n(C)$ is nonzero, and thus there is some nonzero $x\in C$ such that $x$ is an entry of the image of $G(X_1,\ldots ,X_e)$ in $M_n(C)$. By Lemma \ref{lem:Krull} there is a cofinite ideal $L$ of $C$ such that the image of $x$ is nonzero in $C/L$. In particular, we have that the images of $G(X_1,\ldots ,X_s)$ in $M_n(C/L)$ under the composition of maps
$$S\to S/I \to M_n(C)\to M_n(C/L)$$ is nonzero, and so the image of $S/I$ in $M_n(C/L)$ is a finite ring which by construction satisfies every identity from the set $\mathcal{S}$ and does not satisfy the identity $G=0$.
\end{proof}

}
%Let $I$ be the $T$-ideal in $S:=\mathbb{F}_p\{X_1,\ldots, X_s\}$ generated by the identities $F_1=\cdots =F_d=0$.  Then by a result of Belov , the algebra
%$A:=\mathbb{F}_p\{X_1,\ldots, X_s\}/I$ is representable and hence there is a commutative $\mathbb{F}_p$-algebra $C$ such that $A$ embeds as a subalgebra of $M_n(C)$ for some $n\ge 1$. We may replace $A$ by its image in $M_n(C)$.  In addition, we may replace $C$ by the finitely generated subalgebra generated by the entries of a finite set of generators for $A$.  We therefore assume that $C$ is finitely generated and hence noetherian.  By assumption the images of $G_1(X_1,\ldots ,X_s),\ldots ,G_e(X_1,\ldots ,X_s)$ in $A'$ are nonzero in $M_n(C)$ and thus there are nonzero $c_1,\ldots, c_e\in C$ such that $c_i$ is an entry of the image of $G_i(X_1,\ldots ,X_e)$ in $M_n(C)$.   Then by Remark , there is some ideal $L$ with $C/L$ finite such that the images of $c_1,\ldots ,c_e$ in $C/L$ are all nonzero.  In particular, the images of $G_i(X_1,\ldots ,X_e)$ in $M_n(C/L)$ under the composition of maps
%$$S\to S/I \to M_n(C)\to M_n(C/L)$$ are nonzero.  In particular, the image of $S/I$ in $M_n(C/L)$ is a finite ring which by construction satisfies every identity from the set $X$ and does not satisfy the identity $G_i$ for $i=1,\ldots, e$.
%\end{proof}
\begin{remark}
It is interesting to compare Theorem \ref{lem:p} with other algebraic objects.  We note that work by Kle\u\i{man} \cite{Kleiman} on groups and Murski\u\i ~\cite{Murskii} on semigroups shows that these situations are very different and that one does not expect analogues of our results to hold in these settings.
\end{remark}

As mentioned earlier, the most important special case of Question \ref{A} has historically been the case of studying polynomial identity rings for which $[X,Y]=0$ is a non-identity (see, for example, \cite[Chapt. 6]{Her} and references therein). In this case, we can give a more direct proof.

\begin{thm} Let $\mathcal{S}$ be a set of polynomial identities.  If there is a noncommutative ring for which the elements of $\mathcal{S}$ are identities, then there exists a finite noncommutative ring with this property.
\label{thm:Specht}
\end{thm}
\begin{proof}
If $\mathcal{S}$ is empty, we can take $R$ to be the ring of $2\times 2$ matrices over a finite field. Thus we may assume $\mathcal{S}$ is non-empty.
Let $I$ denote the $T$-ideal of $\mathbb{Z}\{x,y\}$ generated by the identities in $\mathcal{S}$. Then by assumption, the ring $R:=\mathbb{Z}\{x,y\}/I$ is not commutative. In particular, the image of $a:=[x,y]$ is nonzero in $R$. Now let $\mathcal{X}$ denote the collection of ideals $J$ of $R$ such that $R/J$ is not a commutative ring (i.e., the ideals that do not contain the image of $a$ in $R$). Then by Zorn's lemma, there exists a maximal element $J_0$ of $\mathcal{X}$.  Then we may replace $R$ by $R/J_0$ and then assume that $R/L$ is commutative whenever $L$ is a nonzero ideal of $R$.  By construction, $R$ is a $\mathbb{Z}$-algebra that is generated by two elements $u$ and $v$ that do not commute such that every nonzero ideal of $R$ contains $[u,v]$; moreover, $R$ satisfies the identities in $\mathcal{S}$. If $R$ is finite, then there is nothing to do.  Thus we may assume that $R$ is infinite. Let $L=R[u,v]R$. Then by construction $L$ is a minimal nonzero two-sided ideal of $R$ and hence it is a simple
$T:=R\otimes_{\mathbb{Z}} R^{\rm op}$-module. Since $\mathcal{S}$ is non-empty, $R$ satisfies a polynomial identity and so by a theorem of Regev \cite[Theorem 1, p. 152]{Regev}, $T$ does too. By minimality of $L$ as a nonzero two-sided ideal of $R$, the annihilator of $L$ as a left $T$-module is a primitive ideal $Q$ of $T$. Then by Kaplansky's theorem \cite[Theorem 6.1.25]{R2}, $T/Q\cong M_n(D)$, where $D$ is a division ring that is finite-dimensional over its center.  Thus $T/Q$ is a finite module over its center and since $T$ is a finitely generated $\mathbb{Z}$-algebra, the center of $T/Q$ is a finitely generated $\mathbb{Z}$-algebra by the Artin-Tate lemma (cf. \cite[\S6.2]{R2}).  In addition, the centre of $T/Q$ is a field and thus by the Nullstellensatz \cite[Theorem 4.19, p. 132]{Eisenbud}, the center of $T/Q$ is a finite field and so $T/Q$ is finite.  Notice that $L=T\cdot [x,y]$ and since $T/Q$ is finite and $Q$ annihilates $L$, we then see that $L$ is necessarily finite.

Since $L$ is finite, there are cofinite left and right ideals $I_1$ and $I_2$ of $R$ such that $I_1\cdot L=L\cdot I_2=(0)$.  These ideals then contain cofinite two-sided ideals and by taking the intersection of these ideals, we see that there is a two-sided ideal $I$ of $R$ such that $R/I$ is a finite ring and such that $IL=LI=(0)$.  Since $R/L$ is a homomorphic image of $\mathbb{Z}[u,v]$, we see that $R/L$ is noetherian and since $L$ is finite, $R$ is both left and right noetherian as well.

Since $R$ is a countable ring we can take an enumeration $r_1,r_2,r_3,\ldots $ of the elements of $R$. Notice that for each $r\in R$, the elements $[u,r]$ and $[v,r]$ lie in $L$.  Thus each element $r\in R$ gives us a map $f_r: \{u,v\}\to L$, given by $f_r(u)=[u,r]$ and $f_r(v)=[v,r]$. Then since $L$ is finite, there are only finitely many maps from $\{u,v\}$ to $L$, so we see that there is some natural number $N$ such that whenever $n> N$, there is some $i\le N$, depending on $n$, such that $f_{r_n}=f_{r_i}$. In particular, one sees that $r_n-r_i$ is central. It follows that $R$ is spanned by $r_1,\ldots ,r_N$ as a module over its center, $Z(R)$. By the Artin-Tate lemma (cf. \cite[\S6.2]{R2}), $Z(R)$ is finitely generated as a $\mathbb{Z}$-algebra and hence it is noetherian. Notice that if $z\in Z(R)$, then multiplication by $z$ induces a self-map of the finite ideal $L$ and hence there is some monic integer polynomial $P(z)$ that annihilates $L$, since $L$ is finite. We claim that $P(z)^n = 0$ in $R$ for some $n$.  To see this, suppose that this is not the case. Then, for each $n$, $P(z)^nR$ is a nonzero ideal of $R$ and by minimality of $L$ as a nonzero ideal, for each $n\ge 1$ there is some $x_n\in R$ such that  $P(z)^n x_n = [u,v]$.  Now let $I_n = \{x\in R\colon P(z)^n x\in L\}$. Then $I_1\subseteq I_2\subseteq I_3\subseteq \cdots $ is an ascending chain of ideals in $R$ and since $R$ is noetherian, there is some $\ell$ such that $I_{\ell}=I_{\ell+1}$.  In particular, $P(z)^{\ell} x_{\ell+1}\in L$. But since $P(z)$ annihilates $L$, this gives $$[u,v]=P(z)^{\ell+1} x_{\ell+1} =P(z)(P(z)^{\ell} x_{\ell+1})\in  P(z)L = (0),$$ a contradiction. Thus there is some $n$ such that $P(z)^n=0$. Since $P(z)^n$ is a monic polynomial with integer coefficients, every $z\in Z(R)$ is integral over the image of $\mathbb{Z}$ in $Z(R)$ and since $Z(R)$ is a finitely generated $\mathbb{Z}$-algebra, we see that $Z(R)$ is a finitely generated $\mathbb{Z}$-module. Since $L$ is finite, there is some positive integer $b$ such that $bL=(0)$. Therefore, the same argument as before shows that there is some $n$ such that $k:=b^n=0$ in $R$.  Thus $Z(R)$ is a finitely generated $\mathbb{Z}/k\mathbb{Z}$-module and hence it is finite. Since $R$ is a finitely generated $Z(R)$-module, $R$ must be finite too. The result follows.
\end{proof}
We now use Theorem \ref{lem:p} to give a coarse classification of the minimal finite noncommutative rings $R$ for which a collection $\mathcal{S}$ of polynomial identities must hold if there exists at least one noncommutative ring for which all identities in $\mathcal{S}$ hold. To do this, we introduce three classes of rings. For each prime $p$, we let
$U_p$ denote the ring of upper-triangular $2\times 2$ matrices with entries in $\mathbb{F}_p$; that is,
\begin{equation}
U_p = \left\{ \left( \begin{array}{cc} a & b \\ 0 & c \end{array}\right) \colon a,b,c\in \mathbb{F}_p\right\}.
\end{equation} Given a prime $p$, an integer $n\ge 2$, and $i\in \{1,\ldots ,n-1\}$, we let
$B_{p,n,i}$ denote the ring
\begin{equation} B_{p,n,i}:= \left\{ \left(\begin{array}{cc} x^{p^i}& y \\ 0 & x \end{array} \right) \colon x,y\in \mathbb{F}_{p^n}\right\}.
\end{equation}  Finally, given a prime $p$, we let $\mathcal{A}_p$ denote the collection of noncommutative rings that are a homomorphic image of a ring of the form
\begin{equation}
\mathbb{Z}\{x,y\}/(I+J_n)\end{equation} for some $n\ge 3$, where $I=(p,x,y)[x,y]\mathbb{Z}\{x,y\} +  \mathbb{Z}\{x,y\}[x,y] (p,x,y)$ and $J_n$ is the ideal $(x,y,p)^n$.
\begin{thm} Let $\mathcal{S}$ be a set of polynomial identities, and suppose there exists a noncommutative ring $R$ that satisfies every identity in $\mathcal{S}$.  Then one of the following must hold:
\begin{enumerate}
\item[(a)] there is a prime $p$ such that $U_p$ satisfies the identities in $\mathcal{S}$;
\item[(b)] there is a prime $p$ and $n\ge 2$ and $i\in \{1,\ldots ,n-1\}$ such that $B_{p,n,i}$ satisfies the identities in $\mathcal{S}$;
\item[(c)] there is a prime $p$ and a ring in $\mathcal{A}_p$ that satisfies the identities in $\mathcal{S}$.
\end{enumerate}
In particular, if there is a noncommutative ring that satisfies the identities of $\mathcal{S}$ then there is a finite noncommutative ring with nonzero nilpotent commutator ideal that satisfies the identities of $\mathcal{S}$.
\label{thm:trichotomy}
\end{thm}
\begin{proof}
Pick $u$ and $v$ in $R$ that do not commute. Then we may replace $R$ by the $\mathbb{Z}$-subalgebra generated by $u$ and $v$ and assume $R$ is two-generated.
By Theorem~\ref{thm:Specht} we may also assume that $R$ is finite. By replacing $R$ by $R/I$, where $I$ is maximal with respect to not being commutative, we may further assume that $R$ has a minimal nonzero two-sided ideal $L$ generated by $[u,v]$.  Moreover, since every proper homomorphic image of $R$ is commutative and since $R$ is finite, there is some prime $p$ such that $p^m R=(0)$.
We now argue via cases.
\vskip 2mm
\emph{Case 1.} $R$ is semiprimitive, that is, $J(R)=(0)$.
\vskip 2mm In this case, by the Artin-Wedderburn theorem \cite[Chapter 14]{R3}, $R$ is isomorphic to a finite product of matrix rings over finite fields.  Since $R$ is not commutative, there is some factor $M_n(F)$ with $n>1$ and $F$ a field of characteristic $p$, which will satisfy every identity that $R$ does.  Since $U_p$ is isomorphic to a subring of $M_n(F)$, we see that $U_p$ satisfies all the identities that $R$ does.
\vskip 2mm
\emph{Case 2.} $J(R)\neq (0)$.
\vskip 2mm
Since $L$ is minimal among nonzero ideals of $R$, $L$ is contained in $J(R)$ and hence $(0)=J(R)L=L^2$, because the Jacobson radical of a finite ring is nilpotent.  Since $R$ has finitely many primitive ideals, which are pairwise comaximal, and since their product is contained in $J(R)$, we then see that there is a unique primitive ideal $P$ of $R$ such that $PL=(0)$; similarly, there is a unique primitive ideal $Q$ of $R$ such that $LQ=(0)$.
\vskip 2mm
\emph{Subcase 2.a.} $P\neq Q$.
\vskip 2mm
%Now if $P\neq Q$ then observe that if $x\in P\cap Q$ is central, then $[x,u]=[x,v]=0$.  On the other hand, if $x$ is nonzero then $Rx$ is a two-sided ideal that contains $L$ and hence $[u,v]=sx$ for some $s\in R$.  Then $[[u,v],u]=[s,u]x=0$ since $x\in Q$ and $[s,u]\in L$.  Similarly, $[[u,v],v]=0$, and so $[u,v]$ is central.  But now since $P$ and $Q$ are distinct primitive ideals of $R$, they are comaximal and hence there is some $y\in P$ such that $1-y\in Q$.  Hence $[y,[u,v]] = y[u,v] + [u,v](1-y-1) = -[u,v]\neq 0$, a contradiction.
%Thus if $P\neq Q$, we have an embedding
%$$\phi: R\to \left(\begin{array}{cc} R/P & L\oplus L \\ 0 & R/Q\end{array}\right),$$ given by
%$$\phi(r) =  \left(\begin{array}{cc} \pi_P(r) & ([r,u],[r,v]) \\ 0 & \pi_Q(r)\end{array}\right).$$
In this case, we can pick $y\in P$ such that $1-y\in Q$.  Then $[[u,v],y]=-y[u,v]+[u,v](1-(1-y)) = [u,v]$.  Let $z=[u,v]$.  Then since $z\in L$ and since $p\in P\cap Q$, we have $pz=z^2=yz=0$ and $zy=z$.  Therefore, the subring $S$ generated by $z$ and $y$ is not commutative and satisfies all the identities that $R$ does.   Notice that $S=S_0\oplus S_1$, where $S_0$ is the image of $\mathbb{Z}[y]$ in $S$ and $S_1=\{0,z,2z,\ldots, (p-1)z\}$.  Let $h(X)$ denote the minimal polynomial of $y$ in $S_0$.  Then since $z(y-1)=0$ and $yz=0$, we see that $h(X)\in (p,(X-1)X)\mathbb{Z}[X]$.  In particular, there is an ideal $I$ of $S_0$ such that $S_0/I\cong \mathbb{F}_p[X]/(X^2-X)$ with an isomorphism that sends $y+I$ to $X+(X^2-X)$, and since $IS_1= S_1 I=(0)$, we see that $J:=I\oplus (0)$ is an ideal of $S$ and so $S/J$ is isomorphic to the three-dimensional $\mathbb{F}_p$-algebra $B$ with generators $s,t$ and with relations
$$s^2=ts=t^2-t=s(1-t)=0.$$
We then have a map $\phi: U_p\to B$ via the map $e_{1,2}\mapsto s$, $e_{2,2}\mapsto t$. This map is easily checked to be an isomorphism. 
\vskip 2mm
\emph{Subcase 2.b.} $P=Q$ and $L$ is not contained in the center of $R$.
\vskip 2mm
In this case, we claim that $pR=(0)$. If not, then since $L$ is minimal among nonzero ideals, we have $pR\supseteq L$ and so $[u,v]=pr$ for some $r\in R$. But since $L$ is not central, we observe that there is some $z$ such that
$[z,[u,v]]\neq 0$. However, this is equal to $p[z,r]$ and $pL=(0)$, a contradiction. Thus $pR=(0)$ and so $R$ is an $\mathbb{F}_p$-algebra. Now we pick $y\in L$ that is not central.  Then $R/P$ is a finite field and hence isomorphic to $\mathbb{F}_{p^n}$ for some $n\ge 1$ and we let $q=p^n$.  We let $x\in R$ be such that $x+P$ is a generator for the multiplicative group of $R/P$. Then since the elements of $P$ annihilate $y$, we see that $[x,y]\neq 0$ since $y$ is non-central and $R$ is generated as an algebra by $x$ and $P$ by construction. Therefore $x^q-x\in P$ and so $$(0)=(x^q-x)Ry=yR(x^q-x).$$  We may replace $R$ by the noncommutative subalgebra generated by $x$ and $y$ if necessary and then $L$ is a simple $R/P$-$R/P$-bimodule generated by $y$ as a bimodule. Thus $L$ is isomorphic to a simple quotient of $R/P\otimes_{\mathbb{F}_p} R/P \cong \mathbb{F}_q^n$ and moreover we have $(x+P)\cdot y \neq y\cdot (x+P)$.   A simple quotient $M$ of $R/P\otimes_{\mathbb{F}_p} R/P$ is isomorphic $\mathbb{F}_q$ and satisfies $(x+P)\cdot v =v\cdot (x^{p^i}+P)$ for every $v\in M$ for some fixed $i\in \{0,1,2,\ldots ,n-1\}$. Since $L$ is not central, we then see that $L\cong \mathbb{F}_q$ and $(x+P)\cdot y =y\cdot (x^{p^i}+P)$ for some fixed $i\in \{1,2,\ldots ,n-1\}$.

Now fix an isomorphism $f:R/P\to \mathbb{F}_q$. Then we claim there is an endomorphism
$\Phi : R\to B_{p,n,i}$ defined on the generators $x$ and $y$ by
$$x\in R\mapsto \left( \begin{array}{cc} f(x^{p^i}+P) & 0\\ 0 & f(x+P)\end{array}\right)$$
and
$$y\in R\mapsto \left( \begin{array}{cc} 0 & 1\\ 0 & 0\end{array}\right).$$
To show that this is an endomorphism, we must show that if $P(X,Y)$ is an element of the free algebra $\mathbb{F}_p\{X,Y\}$ such that $P(x,y)=0$ in $R$ then $P(\Phi(x),\Phi(y))=0$.
One can check that $$\Phi(x)^q-\Phi(x)=\Phi(y)^2=\Phi(x)\Phi(y)-\Phi(y)\Phi(x^{p^i})=0.$$  Thus we may reduce $P(X,Y)$ modulo the ideal $(X^q-X,Y^2, XY-YX^{p^i})$ and we may assume without loss of generality that $P(X,Y)$ is of the form $A(X) + y B(X)$, where $A(X),B(X)\in \mathbb{F}_p[X]$ have degree at most $q-1$.  Now suppose that $A(x)+y B(x)=0$ in $R$ with $A(X), B(X)$ of degree $\le q-1$.  Then left-multiplying by $y$ gives that
$yA(x)=0$.  In particular, since the right annihilator of $y$ is a proper right ideal that contains $P$ and since $R/P$ is a field, we then see that $A(x)\in P$ and so $f(A(x)+P)=f(A(x)^{p^i}+P)=0$.  Thus $\Phi(A(x))=0$.  Thus we may assume that $P(X,Y)$ is of the form $yB(X)$ with $B(X)\in \mathbb{F}_p[X]$.  Then as before, since $B(X)$ annihilates $y$, $B(X)\in P$ and so $\Phi(y)B(\Phi(x))=0$.
Since $\Phi$ is surjective, $B_{p,n,i}$ satisfies all the identities that $R$ does.

\vskip 2mm
\emph{Subcase 2.c.} $P=Q$ and $L$ is central.
In this case, for $z,x\in R$ we have $[z^p,x]={\rm ad}_z^p(x) = {\rm ad}_z^{p-1}([z,x]) = 0$, since $[z,u]\in L$ and elements of $L$ are central.  Then since $R/J(R)$ is a finite product of fields of characteristic $p$, there exists some $m$ such that $u^{p^m}-u$ and $v^{p^m}-v\in J(R)$. Since $p$-th powers are central, we derive that $[u^{p^m}-u,v^{p^m}-v]=[u,v]\neq 0$, and so by considering the subring $S$ of $R$ generated by
$a:=u^{p^m}-u$ and $b:=v^{p^m}-v$, we see that $S/J(S)\cong \mathbb{F}_p$ and $S$ is noncommutative.  In particular, since $S$ is a finite ring, $J(S)^n=(0)$ for some $n\ge 1$ and so we have $(p,a,b)^n=(0)$ in $S$.  Since $S$ is noncommutative and $[a,b]\in J(S)^2$ is nonzero, we see that $n\ge 3$.  By replacing $S$ by a suitable homomorphic image, we may assume that $S$ is noncommutative but that $S/I$ is commutative for all nonzero ideals $I$ of $S$.  We next claim that
$J(S)[a,b]S=S[a,b]J(S)=(0)$.  We only prove $J(S)[a,b]S=(0)$, with the other direction handled in a similar manner. To see this, observe that if $I:=J(S)[a,b]S$ is nonzero then since $S/I$ is commutative, we have
$[a,b]\in J(S)[a,b]S$ and so $$[a,b] = \sum_{i=1}^m x_i [a,b] y_i$$ for some $m\ge 1$, $x_1,\ldots ,x_m\in J(S)$, and $y_1,\ldots ,y_m\in S$.  Now we let $j\ge 1$ denote the smallest positive integer such that
$J(S)^j [a,b]=(0)$.  Then there is some $\theta\in J(S)^{j-1}$ such that $\theta[a,b]\neq 0$.  But now
$$\theta [a,b] = \sum_{i=1}^m (\theta x_i) [a,b] y_i\subseteq J(S)^j [a,b]S=(0),$$ a contradiction.  It follows that $S$ is in the class $\mathcal{A}_p$, which completes the proof.

\end{proof}
\begin{remark}\label{rem:trichotomy}
We point out that the proof of Theorem \ref{thm:trichotomy} in fact shows the following: if $R$ is a finite ring that is not commutative then after a finite set of steps in which at each step we either replace $R$ by a subring or a homomorphic image, we will arrive at one of the finite rings appearing in the statement of Theorem \ref{thm:trichotomy}.  In particular, if $R$ is a \emph{minimal} finite noncommutative ring (that is, a ring with the property that every proper subring and every proper homomorphic image is commutative), then it must appear among one of the three classes of rings appearing in the statement of Theorem \ref{thm:trichotomy}.  In this sense, we consider this result as giving a coarse classification of minimal finite noncommutative rings. We point out, however, that not all the rings that appear in the statement of Theorem \ref{thm:trichotomy} are minimal and it is an interesting problem to give a precise classification of such rings, particularly for the class $\mathcal{A}_p$ with $p$ a prime.  In general, the question of whether a minimal noncommutative ring is necessarily finite appears to be difficult.  For example, one would need to rule out the existence of infinite simple rings with the property that each pair of noncommuting elements generates the entire ring.
\end{remark}
Notice that Theorem \ref{thm:trichotomy} gives a quick proof of the fixed-degree version of Jacobson's $X^n=X$ theorem \cite{J}: if for some $n\ge 2$ there is a noncommutative ring for which the identity $X^n=X$ holds, then the above result shows there is a finite noncommutative ring with a nonzero nilpotent commutator ideal for which the identity $X^n=X$ holds; but if $a$ is a nonzero element in this nilpotent ideal then $0=a(1-a^{n-1})$ and since $a$ is nilpotent, $1-a^{n-1}$ is a unit, so $a=0$, a contradiction.  The same argument applies to show that if $n\ge 2$ then the identity $[X,Y]^n=[X,Y]$ forces a ring to be commutative, which is a special case of a result due to Herstein \cite{HerX}.

%\section{Classification of identities that force commutativity}
%In this section, we use Theorem~\ref{thm:trichotomy} to examine the question: \emph{which polynomial identities $G(X_1,\ldots ,X_s)=0$ have the property that whenever $G=0$ is an identity for a ring $R$, $R$ is necessarily commutative?}

\section{Decidability Procedures}\label{sec:dec}
\label{Algorithm}
In this section we prove Theorem \ref{thm:algorithm} by showing that the question of whether a finite set of identities forces a ring to be commutative is, in fact, decidable and we give an algorithm that always terminates after finitely many steps to make such a decision.  By Theorem \ref{thm:trichotomy} it suffices to check whether there exists a ring from one of the three classes of algebras given in the statement of the theorem for which the identities all simultaneously hold.  We now describe a decision procedure to deal with each of these three classes. Before giving this procedure, we first give some notation.
We let
\begin{equation}
\mathcal{C}_s\subseteq \mathbb{Z}\{X_1,\ldots ,X_s\}
\label{eq:C}
\end{equation}
denote the $\mathbb{Z}$-submodule generated by all monomials $X_1^{i_1}\cdots X_s^{i_s}$ with $i_1,\ldots ,i_s\ge 0$.
Notice that the canonical homomorphism
\begin{equation} \label{eq:Phi}
\Phi: \mathbb{Z}\{X_1,\ldots ,X_s\}\to \mathbb{Z}[X_1,\ldots ,X_s]
\end{equation} with $\Phi(X_i)=X_i$ has the property that the restriction of $\Phi$ to $\mathcal{C}_s$ gives a set bijection between $\mathcal{C}_s$ and the polynomial ring, and given an element $P\in \mathbb{Z}\{X_1,\ldots ,X_s\}$ there is a unique element $\bar{P}\in \mathcal{C}_s$ such that $\Phi(P-\bar{P})=0$.  This map $P\mapsto \bar{P}$ is a transversal of the projection map $\mathbb{Z}\{X_1,\ldots ,X_s\}\to \mathbb{Z}[X_1,\ldots ,X_s]$. In practice, given $P\in \mathbb{Z}\{X_1,\ldots ,X_s\}$, one can compute $\bar{P}$ by simply replacing each monomial that occurs in $P$ by the rearrangement of the letters that puts it in the form $X_1^{i_1}\cdots X_s^{i_s}$.

A key component of our algorithm involves dividing our analysis into two cases: the case when $\Phi(P_i)=0$ for all $i$; and the case when there is some $i$ such that $\Phi(P_i)\neq 0$.  In the latter case,
when there is some $i$ such that $\Phi({P}_i)$ is nonzero, we let $D$ denote the total degree of this nonzero polynomial. By a result of Alon \cite{Alon} there are integers $(n_1,\ldots ,n_s)\in \{0,1,\ldots ,D\}^s$ such that $N:=\Phi({P}_i)(n_1,\ldots ,n_s)$ is a nonzero integer. Then if $P_1,\ldots ,P_m$ are simultaneously identities for a ring $R$, we necessarily have $N=0$ in $R$.  It follows that if $R$ is a ring in one of the three classes given in the statement of Theorem \ref{thm:trichotomy} then $p\mid N$.  An analysis of the above argument shows that if $\bar{P}_1,\ldots ,\bar{P}_m$ are not identically zero then we have an algorithm for determining a finite (possibly empty) set of prime numbers $\{p_1,\ldots ,p_t\}$ such that if  $P_1,\ldots ,P_m$ are identities for a ring $R$ in one of the three classes given in the statement of Theorem \ref{thm:trichotomy} then the associated prime number $p$ must be in $p\in \{p_1,\ldots ,p_t\}$.  In fact, if the integer $|N|$ has prime factorization $p_1^{a_1}\cdots p_t^{a_t}$ then the characteristic of $R$ must be a divisor of one of the elements from $\{p_1^{a_1},\ldots ,p_t^{a_t}\}$.

For the remainder of this section, we assume that we are given $P_1,\ldots ,P_m \in \mathbb{Z}\{X_1,\ldots ,X_s\}$ and we will answer the question: ``\emph{Is there a noncommutative ring for which $P_1,\ldots ,P_m$ are all identities?}''  We do so by considering each of the three classes of rings in the statement of Theorem \ref{thm:trichotomy} separately.
\subsection{Decision procedures for $U_p$}
We begin with the simplest case, which is to decide whether there is a prime $p$ such that $P_1,\ldots ,P_m$ are all identities for the algebra $U_p$.

\begin{lemma} Let $P_1,\ldots, P_m\in \mathbb{Z}\{X_1,\ldots ,X_s\}$.  Then it is decidable whether or not there is some prime $p$ such that $P_1,\ldots ,P_m$ are identities for $U_p$.
\end{lemma}
\begin{proof}
We let $\mathcal{X}$ denote the subset of $M_2(\mathbb{Z})$ consisting of the eight upper-triangular matrices with $\{0,1\}$-entries.  Then for each $s$-tuple $(A_1,\ldots ,A_s)\in \mathcal{X}^s$ we compute the matrices
$P_i(A_1,\ldots ,A_s)$ for $i=1,\ldots ,m$.  If $P_i(A_1,\ldots ,A_s)=0$ for every $i\in \{1,\ldots ,m\}$ and every
$(A_1,\ldots ,A_s)\in \mathcal{X}^s$ then since the $P_i$ are integer polynomials and since the image of $\mathcal{X}$ in $U_2$ is all of $U_2$, we have that $P_1,\ldots ,P_m$ are identities for $U_2$.  Alternatively, there is some $i$ and some $(A_1,\ldots ,A_s)\in \mathcal{X}^s$ such that $P_i(A_1,\ldots ,A_s)$ is a nonzero integer.  We compute the gcd, $d$, of the entries of this matrix.  If $d$ is equal to one then there cannot exist a prime $p$ such that $P_1,\ldots ,P_m$ are identities for $U_p$, and so we may assume that the gcd is strictly greater than one.  We then compute the primes $p_1,\ldots ,p_s$ that divide $d$.  Then if  $P_1,\ldots ,P_m$ are identities for $U_p$ for some prime $p$ then $p$ is necessarily in $\{p_1,\ldots ,p_s\}$.  Then for each $p\in \{p_1,\ldots ,p_s\}$, it can be checked whether or not $P_1,\ldots ,P_m$ are identities for $U_p$ by simply taking each $s$-tuple of elements from $U_p$ and verifying whether the evaluations of $P_1,\ldots ,P_m$ at this $s$-tuple are zero in $U_p$; if all possible evaluations are zero, then $P_1,\ldots ,P_m$ are identities for $U_p$. Since $U_p$ is a finite ring and there are finitely many primes $p$ to check, this process terminates and so we have a decision procedure to determine whether or not $P_1,\ldots ,P_m$ are identities for $U_p$ for some prime $p$.
\end{proof}
\subsection{Decision procedures for the algebras $B_{p,n,i}$}
We now show how one can check whether $P_1,\ldots ,P_m$ are identities for some algebra of the form $B_{p,n,i}$.  To proceed, we require a lemma.
\begin{lemma} Let $Q(X_1,\ldots ,X_s)\in \mathcal{C}_s$ be nonzero and let $q$ be a power of a prime $p$.  If $Q(X_1,\ldots ,X_s)$ is an identity for the finite field $\mathbb{F}_q$, then $\Phi(Q)\in (p,X_1^q-X_1,\ldots ,X_s^q-X_s)\mathbb{Z}[X_1,\ldots ,X_s]$ and if the reduction of $\Phi(Q)$ mod $p$ is not identically zero then the reduction has total degree at least $q$. 
\label{lem:Alon}
\end{lemma}
\begin{proof} Since $\mathbb{F}_q$ is commutative, we may replace $Q$ by $\Phi(Q)$, its image in the polynomial ring $\mathbb{Z}[X_1,\ldots ,X_s]$, and assume it is a nonzero (commutative) polynomial.  If $Q$ is identically zero mod $p$ there is nothing to prove, so we assume that this is not the case and let $d$ denote the total degree of the reduction of $Q$ mod $p$. 
Then there exist $t_1,\ldots ,t_s$ with $$\sum t_i=d$$ such that the coefficient of $X_1^{t_1}\cdots X_s^{t_s}$ is nonzero.  It $t_i<q$ for $i=1,\ldots ,s$ then by Alon's combinatorial Nullstellensatz \cite[Theorem 1.2]{Alon} there exist 
$\alpha_1,\ldots ,\alpha_s\in \mathbb{F}_q$ such that $Q(\alpha_1,\ldots ,\alpha_s)\neq 0$ and so we see that the total degree of $Q$ must be at least $q$.  To complete the proof, we must show that $Q\in  (p,X_1^q-
X_1,\ldots ,X_s^q-X_s)$.  Since each $X_i^q-X_i$ is an identity for $\mathbb{F}_q$, we may reduce $Q$ modulo the ideal $$ (p,X_1^q-X_1,\ldots ,X_s^q-X_s)$$ and assume that it has degree at most $q-1$ in each variable $X_i$.  If 
$Q$ is nonzero then there exist  $t_1,\ldots ,t_s$ with $t_i<q$ for $i=1,\ldots ,s$ such that the coefficient of $X_1^{t_1}\cdots X_s^{t_s}$ is nonzero. But a second application of \cite[Theorem 1.2]{Alon} shows there exist $
\alpha_1,\ldots ,\alpha_s\in \mathbb{F}_q$ such that $Q(\alpha_1,\ldots ,\alpha_s)\neq 0$, a contradiction.
\end{proof}
The following lemma requires the use of Cartier operators.  We let $p$ be a prime number.
Then $\mathbb{F}_p[X_1,\ldots ,X_s]$ is a free $\mathbb{F}_p[X_1^p,\ldots ,X_s^p]$-module with basis 
$X_1^{j_1}\cdots X_s^{j_s}$ with $(j_1,\ldots ,j_s)\in \{0,\ldots ,p-1\}^s$, and every element of $\mathbb{F}_p[X_1^p,\ldots ,X_s^p]$ is the $p$-th power of some element of $\mathbb{F}_p[X_1,\ldots ,X_s]$.
In particular, for each $s$-tuple of integers $(j_1,\ldots ,j_s)\in \{0,\ldots ,p-1\}^s$, we define maps 
\begin{equation}
\Lambda_{j_1,\ldots ,j_s}:{{\mathbb F}}_p[X_1,\ldots ,X_s]\to {{\mathbb F}}_p[X_1,\ldots ,X_s],
\end{equation}
which are the operators uniquely defined by
\begin{equation}
P(X_1,\ldots ,X_s) = \sum_{j_1=0}^{p-1}\cdots \sum_{j_s=0}^{p-1} X_1^{j_1}\cdots X_s^{j_s} \Lambda_{j_1,\ldots ,j_s}(P(X_1,\ldots ,X_s))^{p},
\end{equation}
for $P(X_1,\ldots ,X_s)$ in ${{\mathbb F}}_p[X_1,\ldots ,X_s]$.
Observe that \begin{equation}
\label{eq:Cart}\Lambda_{j_1,\ldots ,j_s}(A+B^{p} C) = \Lambda_{j_1,\ldots ,j_s}(A) + B\Lambda_{j_1,\ldots ,j_s}(C)
\end{equation}
for $A,B,C$ in ${{\mathbb F}}_p[X_1,\ldots ,X_s]$.
In addition, if $P$ is a polynomial in $\mathbb{F}_p[X_1,\ldots ,X_s]$ then
$P$ is the zero polynomial if and only if $ \Lambda_{j_1,\ldots ,j_s}(P)=0$ for all $(j_1,\ldots ,j_s)\in \{0,\ldots ,p-1\}^s$.
For us, we shall be considering linear combinations of polynomials of the form $A^{p^k} B$.  In this case, if $B$ has degree strictly less than $p^k$ then if $\Omega$ is a $k$-fold composition of Cartier operators then
$\Omega(A^{p^k} B) = A\Omega(B)$ and moreover $\Omega(B)$ is a coefficient (possibly zero) of some monomial occurring in $B$.

\begin{lemma} Let $A_0,\ldots ,A_t$ and $B_0,\ldots ,B_t$ be polynomials in $\mathbb{Z}[X_1,\ldots ,X_s]$ with $A_0,\ldots ,A_t$ linearly independent over $\mathbb{Z}$ and let $\alpha_0,\ldots ,\alpha_t\in \{1,\ldots ,s\}$.  Then we can decide whether or not there exists a triple $(p,n,k)$, with $p$ a prime, $n\ge 2$ an integer, and $k\in \{1,\ldots, \lfloor n/2\rfloor\}$, such that $\sum_{i=0}^t (X_{\alpha_i}^{p^k}-X_{\alpha_i}) A_i^{p^k} B_i$ is in the ideal $(p, X_1^{p^n}-X,\ldots , X_s^{p^n}-X_s)\mathbb{Z}[X_1,\ldots, X_s]$.
\label{lem:At}
 \end{lemma}
\begin{proof}

Let $\kappa$ be the maximum of the degrees of $A_0,\ldots ,A_t, B_0,\ldots ,B_t$.
Suppose that $\sum_{i=0}^t  (X_{\alpha_i}^{p^k}-X_{\alpha_i})  A_i^{p^k} B_i$ is in $(p, X_1^{p^n}-X,\ldots , X_s^{p^n}-X_s)\mathbb{Z}[X_1,\ldots, X_s]$.
We first consider the case when $p^k > \kappa+2$. Therefore, after reducing mod $p$, we need to determine whether
$\sum_{i=0}^t  (X_{\alpha_i}^{p^k}-X_{\alpha_i})  A_i^{p^k} B_i\in (X_1^{p^n}-X,\ldots , X_s^{p^n}-X_s)\mathbb{F}_p[X_1,\ldots, X_s]$. Then since the total degree is at most
$\kappa p^k + \kappa + p^k < p^{2k}  \le p^n$, we see from Lemma \ref{lem:Alon} that this is the case if and only if
$\sum_{i=0}^t  (X_{\alpha_i}^{p^k}-X_{\alpha_i})  A_i^{p^k} B_i$ is identically zero, when regarded as a polynomial with coefficients in $\mathbb{F}_p$.

Then since $B_i X_{\alpha_i}$ has total degree strictly less than $p^k$, if $\Omega$ is a $k$-fold composition of Cartier operators then $\Omega(B_i)$ and $\Omega(B_i X_{\alpha_i})$ are the coefficients of some fixed monomials in $B_i$ and $B_i X_{\alpha_i}$ respectively.  Thus we see that $\sum_{i=0}^t  (X_{\alpha_i}^{p^k}-X_{\alpha_i})  A_i^{p^k} B_i$ is identically zero mod $p$ if and only if
$$\sum_{i=0}^t  X_{\alpha_i}A_i \Omega(B_i) - A_i \Omega(X_{\alpha_i}B_i)$$ is identically zero for every $k$-fold composition of Cartier operators when we work over $\mathbb{F}_p$.  Since the total degrees of $B_i$ and $B_iX_{\alpha_i}$ are less than $p^k$, we can obtain each coefficient by applying $k$-fold compositions of Cartier operators, and so if we let $\lambda_{i;j_1,\ldots ,j_s}$ denote the coefficient of $X_1^{j_1}\cdots X_s^{j_s}$ in $B_i$ then we see that this is equivalent to
$$\sum_{i=0}^t  X_{\alpha_i}A_i \lambda_{i;j_1,\ldots ,j_s} - A_i \lambda_{i;j_1-\delta_{1,\alpha_i},\ldots ,j_s-\delta_{s,\alpha_i}}$$ being identically zero mod $p$ for each $s$-tuple $(j_1,\ldots ,j_s)$ with $\sum j_i\le \kappa$, where $\delta_{i,j}$ is the Kronecker delta function.

Now for each such $s$-tuple we can compute the gcd of the coefficients of $$\sum_{i=0}^t  X_{\alpha_i}A_i \lambda_{i;j_1,\ldots ,j_s} - A_i \lambda_{i;j_1-\delta_{1,\alpha_i},\ldots ,j_s-\delta_{s,\alpha_i}}$$ and by taking the gcd over each of the gcds produced for each $s$-tuples we can compute a natural number $N$ with the property that for a prime $p$ and integers $k\ge 1$ and $n\ge 2k$ such that $p^k> \kappa+2$ and 
$\sum_{i=0}^t  (X_{\alpha_i}^{p^k}-X_{\alpha_i})  A_i^{p^k} B_i$ is in $$(p, X_1^{p^n}-X,\ldots , X_s^{p^n}-X_s)\mathbb{Z}[X_1,\ldots, X_s]$$ if and only if $p\mid N$.  In particular, if there is some prime $p$ that divides $N$ then there exists a triple $(p,n,k)$ with $k\le n/2$ and $k\ge 1$ such that $\sum_{i=0}^t (X_{\alpha_i}^{p^k}-X_{\alpha_i}) A_i^{p^k} B_i$ is in $(p, X_1^{p^n}-X,\ldots , X_s^{p^n}-X_s)\mathbb{Z}[X_1,\ldots, X_s]$.  If, on the other hand, $N=1$ then we know that if
$\sum_{i=0}^t (X_{\alpha_i}^{p^k}-X_{\alpha_i}) A_i^{p^k} B_i$  is in $(p, X_1^{p^n}-X,\ldots , X_s^{p^n}-X_s)\mathbb{Z}[X_1,\ldots, X_s]$ then $p^k \le \kappa+2$.  Since there are only finitely many pairs $(p,k)$ with $p$ prime and $k\ge 1$ such that $p^k \le \kappa+2$.  We have reduced our analysis to considering triples $(p,n,k)$ with $p^k \le \kappa+2$.
Notice that the total degree of $\sum_{i=0}^t  (X_{\alpha_i}^{p^k}-X_{\alpha_i})  A_i^{p^k} B_i$ is then at most
$\kappa^2+ 4\kappa + 2$ and so if $p^n >\kappa^2+ 4\kappa + 2$ then by Lemma \ref{lem:Alon}, $\sum_{i=0}^t (X_{\alpha_i}^{p^k}-X_{\alpha_i}) A_i^{p^k} B_i$ is in $(p, X_1^{p^n}-X,\ldots , X_s^{p^n}-X_s)\mathbb{Z}[X_1,\ldots, X_s]$ if and only if it is identically zero mod $p$; in particular, this condition is independent of $n$ in this case and we can check this for the finite set of $p$ with $p\le \kappa+2$.  Finally, if $p^k\le \kappa+2$ and $p^n \le \kappa^2+4\kappa+2$ then $(p,n,k)$ lies in a finite set and we can check these cases on a case-by-case basis via computation.
\end{proof}
We can now describe the algorithm for deciding whether $$P_1=\cdots =P_m=0$$ are identities for a noncommutative ring of the form $B_{p,n,\ell}$.  We let $J$ denote the commutator ideal of $\mathbb{Z}\{X_1,\ldots ,X_s\}$ generated by $[X_i,X_j]$ with $i\neq j$.
We argue via cases. 
\vskip 2mm
\emph{Case I.} $\bar{P}_1=\cdots =\bar{P}_m=0$.  

Then in this case we can compute
$P_1,\ldots ,P_m$ mod $J^2$ and each $Q\in \{P_1,\ldots ,P_m\}$ mod $J^2$ is of the form
$$Q:=\sum_{1\le i< j \le s} \sum_{k=1}^{m_{i,j}} A_{i,j,k} [X_i,X_j] C_{i,j,k},$$ where $A_{i,j,k},C_{i,j,k}\in \mathcal{C}_s$, and for each pair $(i,j)$ we have
$\{A_{i,j,k}\colon k\le m_{i,j}\}$ is linearly independent over $\mathbb{Z}$ and $\{C_{i,j,k}\colon k\le m_{i,j}\}$ is linearly independent over $\mathbb{Z}$; moreover, it is not difficult to compute these expressions mod $J^2$.  Since each element of $J^2$ is an identity for $B_{p,n,\ell}$, we see that each such $Q$ is an identity for $B_{p,n,\ell}$ if and only if each of $P_1,\ldots ,P_m$ are identities for $B_{p,n,\ell}$.

Now if $Q$ is an identity for a ring $B_{p,n,\ell}$ then since
$X [Y,Z] = [Y,Z] X^{p^{\ell}}$ is also an identity for $B_{p,n,\ell}$, we have
$$\sum_{1\le i< j \le s} \sum_{k=1}^{m_{i,j}}  [X_i,X_j] A_{i,j,k}^{p^{\ell}}C_{i,j,k}=0$$ is an identity for $B_{p,n,\ell}$ and since the square of the commutator ideal is zero, we may replace each $A_{i,j,k}^{p^{\ell}}C_{i,j,k}$ by their images in the commutative polynomial ring over $\mathbb{F}_p$.  Then we consider the commutative polynomial ring $\mathbb{F}_p[U_1,\ldots ,U_s,V_1,\ldots ,V_s]$.  Then specializing $X_i$ at the element
\[\left(\begin{array}{cc} U_i^{p^{\ell}} & V_i \\ 0 & U_i\end{array} \right), \] we see that an element of the form
$$\sum_{1\le i< j \le s} \sum_{k=1}^{m_{i,j}}  [X_i,X_j] A_{i,j,k}^{p^{\ell}}C_{i,j,k}=0$$ is an identity for $B_{p,n,\ell}$ if and only if
$$H:=\sum_{1\le i< j \le s} \sum_{k=1}^{m_{i,j}} (U_i^{p^{\ell}} V_j +U_j V_i -U_j^{p^{\ell}} V_i - U_i V_j) A_{i,j,k}^{p^{\ell}}(U_1,\ldots ,U_s)C_{i,j,k}(U_1,\ldots ,U_s) $$ is an identity for $\mathbb{F}_{p^n}$.  We let $$\kappa:=\max\{{\rm deg}(A_{i,j,k},C_{i,j,k})\colon i<j, 1\le k\le m_{i,j}\}.$$

We first consider the subcase when $\ell \le n/2$.  In this case, if $p^{n/2}/4> \kappa$ then
the total degree of $H$ is at most $$(p^{n/2}+1)\kappa + p^{n/2}\le p^n/2 + p^{n/2} < p^n,$$ since $n\ge 2$.
It then follows from Lemma \ref{lem:Alon} that the polynomial $H$ must be identically zero after we identify it with its image in $\mathbb{F}_p[U_1,\ldots ,U_s,V_1,\ldots ,V_s]$.  In particular, taking the coefficient of $V_j$, we see that
$$H_j:=\sum_{i\neq j} \sum_{k=1}^{m_{i,j}} (-1)^{\chi(i,j)}(U_i^{p^{\ell}}-U_i) A_{i,j,k}^{p^{\ell}}(U_1,\ldots ,U_s)C_{i,j,k}(U_1,\ldots ,U_s)$$ must be identically zero, when viewed as a commutative polynomial with coefficients in $\mathbb{F}_p$, where $\chi(i,j)$ is $1$ if $i>j$ and is $0$ otherwise.  Thus we have reduced the problem to deciding whether $H_1,\ldots ,H_s$ are zero mod $p$, and by Lemma \ref{lem:At} we can decide whether there exists a prime $p$ for which this occurs.  On the other hand, there are only finitely many triples $(p,n,\ell)$ with $p$ prime, $\ell \le n/2$ and $p^{n/2} < 4\kappa$ and we can check on a case-by-case basis whether $P_1,\ldots ,P_m$ are identities for the algebras $B_{p,n,\ell}$ by evaluating them at all $s$-tuples of elements in these algebras and checking whether the results are always zero.  

The second subcase is when we have a triple $(p,n,\ell)$ with
$\ell>n/2$.  Then applying the $\mathbb{F}_{p^n}$ field automorphism given by $x\mapsto x^{p^{n-\ell}}$ to our expression for $H$ and using the fact that $a^{p^n}=a$ in $\mathbb{F}_{p^n}$, we see this is the case if and only if
$$H':=\sum_{1\le i< j \le s} \sum_{k=1}^{m_{i,j}} (U_i V_j^{p^{n-\ell}} +U_j^{p^{n-\ell}} V_i^{p^{n-\ell}} -U_jV_i^{p^{n-\ell}} - U_i^{p^{n-\ell}} V_j^{p^{n-\ell}}) A_{i,j,k}C_{i,j,k}^{p^{n-\ell} }$$ is an identity for $\mathbb{F}_{p^n}$, where the $A_{i,j,k}$ and $C_{i,j,k}$ are polynomials in the variables $U_1,\ldots ,U_s$.  Moreover, since the $V_i$'s are indeterminates and since the map $x\mapsto x^{p^{n-\ell}}$ is bijective on $\mathbb{F}_{p^n}$, this is the case if and only if
 $$H'':=\sum_{1\le i< j \le s} \sum_{k=1}^{m_{i,j}} (U_i V_j +U_j^{p^{n-\ell}} V_i -U_jV_i - U_i^{p^{n-\ell}} V_j )A_{i,j,k}C_{i,j,k}^{p^{n-\ell} }$$
is an identity for $\mathbb{F}_{p^n}$.  Since $n-\ell\le n/2$, we can now handle this in a completely symmetric manner as we handled the first case when $\ell\le n/2$.  Thus this case is decidable.

\emph{Case II.} $\bar{P}_i$ is nonzero for some $i$.

 In this case, we compute the gcd $d_i$ of the 
coefficients of the monomials occurring in each $\bar{P}_i$ for $i=1,\ldots ,s$. By assumption at least one $d_i$ is 
nonzero and so if we let $\mathcal{S}$ denote the set of primes $p$ which divide $\gcd(d_1,\ldots ,d_t)$ then $\mathcal{S}$ is a finite set.  To check whether there is some prime $p\in \mathcal{S}$ and some integers $n$ and $\ell$ such that $P_1,\ldots ,P_m$ are identities for $B_{p,n,\ell}$, we simply use the procedure given in Case 1 above for these particular primes, since modulo $p$ we have $\bar{P}_1=\cdots =\bar{P}_s=0$.  Now for $p\not\in \mathcal{S}$ we have that there is some $Q\in\{ \bar{P}_1,\ldots ,\bar{P}_s\}$ that is not identically zero mod $p$.  Then if $P_1,\ldots ,P_m$ are identities for $B_{p,n,\ell}$ then $Q$ must be an identity for $B_{p,n,i}/([B_{p,n,i},B_{p,n,i}])\cong \mathbb{F}_{p^n}$. In particular, $Q\in (p, X_1^q-X,\ldots ,X_s^q-X)\subseteq \mathbb{Z}[X_1,\ldots ,X_s]$ and is identically zero mod $p$ if the total degree of $Q$ is strictly less than $p^n$ by Lemma \ref{lem:Alon}.  Thus if we let $D_{p,Q}$ be the total degree of $Q$ mod $p$, then $D_{p,Q}$ is equal to the total degree of $Q$ for all but a finite computable set of primes.  Then if $D_{p,Q}< p^n$, then $Q$ is not identically zero mod $p$ and so by Lemma \ref{lem:Alon}, $Q$ is not an identity for $\mathbb{F}_{p^n}$.  Thus we may consider the pairs $(p,n)$ with $D_{p,Q}\ge p^n$.  But there are only finitely many $n$ and $p$ not in $\mathcal{S}$ such that $D_{p,Q}\ge p^n$ and moreover it is easy to compute all such pairs $(p,n)$.  Once we have computed all eligible pairs $(p,n)$, they give rise to a finite number of eligible triples $(p,n,\ell)$ and we can again then test whether $P_1,\ldots ,P_m$ are identities for this finite set of allowable algebras $B_{p,n,\ell}$ via finitely many computations.

\subsection{Decision procedures for algebras in the class $\mathcal{A}_p$}
We let $J$ denote the commutator ideal of $\mathbb{Z}\{X_1,\ldots ,X_s\}$ generated by all commutators $[X_i,X_j]$ and we let
$I$ denote the sum of $J^2$ and the ideal generated by the elements $[[X_i,X_j],X_k]$.  Then all elements of $I$ are identities for algebras in $\mathcal{A}_p$ and so we first reduce $P_1,\ldots ,P_m$ mod $I$ and we may assume that we have
\begin{equation}\label{eq:Hk}
P_k=H_k + \sum_{i<j} A_{i,j,k} [X_i,X_j],\end{equation} where $H_k, A_{i,j,k}\in \mathcal{C}_s$, and where $\mathcal{C}_s$ is defined as in Equation (\ref{eq:C})

We now give an overview of the procedure we use to test whether there is a ring in the class $\mathcal{A}_p$ for which $P_1,\dots ,P_m$ are all simultaneously identities.

\vskip 2mm
\begin{enumerate}
\item[Step 1.] We first compute $\bar{P}_1,\ldots ,\bar{P}_s$.  If these are all zero, we go to Step 3; otherwise, we go to Step 2.
\item[Step 2.] Use the theory of Gr\"obner-Shirshov bases to decide whether there is an algebra in $\mathcal{A}_p$ for which $P_1,\ldots ,P_m$ are all identities and stop.
\item[Step 3.] Use Lemma \ref{lem:comm} and the procedure described afterwards to decide whether there is an algebra in $\mathcal{A}_p$ for which $P_1,\ldots ,P_m$ are all identities and stop.
\end{enumerate}

The easier case is the third step in the procedure above, which we describe now.  For the following result, we let $[S,S]$ denote the commutator ideal of a ring $S$ and we let $[[S,S],S]$ denote the ideal generated by all elements $[a,b]$ with $a\in [S,S]$ and $b\in S$.  
\begin{lemma} Let $p$ be a prime and suppose that $P_1,\ldots ,P_m\in S:= \mathbb{Z}\{X_1,\ldots ,X_s\}$ are polynomials with $\bar{P}_1=\cdots =\bar{P}_m=0$. Then there exists $R\in \mathcal{A}_p$ for which $P_1,\ldots ,P_m$ are identities for $R$ if and only if $$P_1,\ldots ,P_m \in pS + [S,S]^2+ [[S,S],S] + (X_1^p-X_1,\ldots ,X_s^p-X_s)[S,S].$$
\label{lem:comm}
\end{lemma}
\begin{proof}
Observe that every element of the ideal
$$pS + [S,S]^2+ [[S,S],S] + (X_1^p-X_1,\ldots ,X_s^p-X_s)[S,S]$$ is an identity for the noncommutative ring
$$\mathbb{F}_p\{X,Y\}/(X,Y)^3\in \mathcal{A}_p$$ and so it suffices to prove that the converse holds.

Since each $\bar{P}_i=0$ we have that $P_i$ is in the commutator ideal for $i=1,\ldots , m$.  Then since every element of $\mathbb{Z}\{X_1,\ldots ,X_s\}$ is congruent to an element of $\mathcal{C}_s$ modulo the commutator ideal, we see that, modulo $J:=[[S,S],S]+[S,S]^2$, we can write
$$P_k \equiv \sum_{1\le i<j\le s} Q_{i,j,k} [X_i,X_j]~(\bmod ~J)$$ with each $Q_{i,j,k}\in \mathcal{C}_s$.  Since every element of $J$ is an identity for every ring in $\mathcal{A}_p$, we may assume that $P_k$ is in fact equal to
$$ \sum_{1\le i<j\le s} Q_{i,j,k} [X_i,X_j]$$ for $k=1,\ldots ,m$.  Now if there exist $i,j,k$ and integers $n_1,\ldots ,n_s$ such that $c:=Q_{i,j,k}(n_1,\ldots ,n_s)$ is not a multiple of $p$, then for $R\in \mathcal{A}_p$ there exist $x,y\in J(R)$ with $[x,y]\neq 0$ and $J(R)[x,y]=[x,y]J(R)=(0)$.  Then if we evaluate
$P_k$ at $$(X_1,\ldots ,X_s)=(n_1,n_2,\ldots ,n_i+x,\ldots ,n_j+y,\ldots ,n_s),$$ we obtain
$c[x,y]$, since $x[x,y]=y[x,y]=0$ in $R$.  But this is a contradiction since $[x,y]\neq 0$ and $P_k$ is assumed to be an identity for $R$.  It follows that $Q_{i,j,k}(n_1,\ldots, n_s)$ is a multiple of $p$ for all $i,j,k$ and integers $n_1,\ldots ,n_s$.
In particular, $\Phi(Q_{i,j,k})$ is an identity for $\mathbb{F}_p$ and so by Lemma \ref{lem:Alon} it is in the ideal
$(p,X_1^p-X_1,\ldots ,X_s^p-X_s)$.  Since
$Q_{i,j,k}\equiv \bar{Q}_{i,j,k}~(\bmod ~[S,S])$, we see that
$$P_1,\ldots ,P_m\in pS + [S,S]^2+ [[S,S],S] + (X_1^p-X_1,\ldots ,X_s^p-X_s)[S,S],$$
as required.
\end{proof}

We now see how we can determine whether there is a prime number $p$ and an algebra $R$ in the class $\mathcal{A}_p$ for which $P_1,\ldots ,P_m$ are all identities when $\bar{P}_1=\cdots \bar{P}_m=0$.  We claim that this is the case if and only if there is a prime $p$ for which each $A_{i,j,k}$ is an identity for $\mathbb{F}_p$, where the $A_{i,j,k}$ are as in Equation (\ref{eq:Hk}).  To see this, observe that using Equation (\ref{eq:Hk}) and the assumption that $\bar{P}_k=0$ for all $k$, we have
$P_k= \sum_{i<j} A_{i,j,k} [X_i,X_j]$ where $A_{i,j,k}\in \mathcal{C}_s$.  Now suppose that there is some $A_{i,j,k}$ that is not an identity for $\mathbb{F}_p$.  Then there are integers $n_1,\ldots ,n_s$ such that $p\nmid A_{i,j,k}(n_1,\ldots ,n_s)$.
We now have that an algebra $R\in \mathcal{A}_p$ is generated by elements $x$ and $y$ in the Jacobson radical of $R$ and we set $a_{\ell}=n_{\ell}$ for $\ell\neq i,j$ and we set $a_i=n_i+x$ and $a_j=n_j+y$.
Then $P_k(a_1,\ldots ,a_s)=A_{i,j,k}(n_1,\ldots ,n_s)[x,y]\neq 0$ in $R$, since $R$ is noncommutative and of characteristic a power of $p$.  Thus $P_k$ cannot be an identity for an algebra in $\mathcal{A}_p$.  Conversely, if each $A_{i,j,k}$ is an identity for $\mathbb{F}_p$ then by Lemma \ref{lem:Alon}, each $\Phi(A_{i,j,k})\in (p,X_1^p-X_1,\ldots ,X_s^p-X_s)\mathbb{Z}[X_1,\ldots ,X_s]$.  In particular, by Lemma \ref{lem:comm} we see that there exists $R\in \mathcal{A}_p$ for which $P_1,\ldots ,P_m$ are identities for $R$.

Thus we have reduced the analysis in the case that $\bar{P}_1,\ldots ,\bar{P}_s$ are all identically zero to the question of whether the polynomials $A_{i,j,k}$ computed above are all simultaneously identities for some $\mathbb{F}_p$.  If the $A_{i,j,k}$ are identically zero, it is immediate that all primes $p$ work; on the other hand, if some $A_{i,j,k}$ is nonzero we let $D$ denote its total degree and we let $D_p$ denote the total degree of its reduction modulo a prime $p$, then $D_p=D$ for all but a finite computable set of primes. Then by Lemma \ref{lem:Alon} if this $A_{i,j,k}$ is an identity for $\mathbb{F}_p$ then $p\le D_p$.  Consequently, we can now handle the case when $\bar{P}_1,\ldots ,\bar{P}_s$ are all identically zero: we compute the finite set of primes $p$ with $p\le D_p$ and by evaluating each polynomial $A_{i,j,k}$ at each of the $p^s$ $s$-tuples of elements of $\mathbb{F}_p$, we can determine whether there is some prime $p$ such that the $A_{i,j,k}$ are all simultaneously identities for $\mathbb{F}_p$.  Thus it now suffices to show what to do in the case when $\bar{P}_1,\ldots ,\bar{P}_s$ are not all identically zero.

For the remaining case, we make use of Gr\"obner-Shirshov bases. The main power of Gr\"obner-Shirshov bases is that they allow one to test ideal membership.  In particular, if one has a finite Gr\"obner-Shirshov basis for an ideal in a finitely generated ring $R$ then one can decide whether the ring is commutative by testing whether the commutators of all generators have zero image in the ring.  There are, however, two potential pitfalls that arise when working in the context of rings satisfying a certain set of polynomial identities.  The first problem is that one requires all specializations of the identities to be zero and so one cannot guarantee that one is working with a finitely generated ideal; the second issue is that for noncommutative algebras there are no guarantees that the Gr\"obner-Shirshov algorithm terminates.  Thankfully, we are able to get around both of these problems in this setting.

We first give a brief overview of how to use Gr\"obner-Shirshov bases in general.  We let $C$ be a finitely generated commutative ring and we let $I$ be a finitely generated two-sided ideal of the free $C$-algebra $R:=C\{x_1,\ldots ,x_s\}$. 
For our purposes, we will assume in what follows that every ideal of $C$ is principal, which will hold in the setting we use. We put a degree lexicographic order $\preceq$ on the monomials $x_1,\ldots ,x_s$ by fixing some order on the elements of $\{x_1,\ldots ,x_s\}$. Given an ideal $I$ we then have a procedure (which need not terminate) to produce a Gr\"obner-Shirshov basis for an ideal $I$ in the algebra $C\{x_1,\ldots ,x_s\}$.  Then given a nonzero element $g\in R$ we have an \emph{initial term}, which is $c\cdot w$ with $c\in C\setminus \{0\}$ and $w\in \{x_1,\ldots ,x_s\}^*$ such that $g-c\cdot w$ is a $C$-linear combination of words in $\{x_1,\ldots ,x_s\}^*$ that are strictly less than $w$ with respect to the order $\prec$.  Given a nonzero element $f\in C\{x_1,\ldots ,x_s\}$, we let 
\begin{equation}
\label{eq:inf}
{\rm in}(f)\in (C\setminus \{0\}) \{x_1,\ldots ,x_s\}^*\end{equation} denote the initial term of $f$.

Given a finite set of generators $f_1,\ldots ,f_d$ for the ideal $I$ there is a procedure (see Bokut and Chen \cite{BC} and Mikhalev, Zolotykh \cite{MZ1,MZ2}) that produces a possibly infinite set of generators $h_1,h_2, \ldots $ for $I$ with the following properties:
\begin{enumerate}
\item ${\rm in}(h_i) \not\in (C\setminus \{0\}) \{x_1,\ldots ,x_s\}^* {\rm in}(h_j)\{x_1,\ldots ,x_s\}^*$ for $i\neq j$;
\item if $h\in I$ is nonzero then there is some $j$ such that $${\rm in}(h)\in (C\setminus \{0\}) \{x_1,\ldots ,x_s\}^* {\rm in}(h_j)\{x_1,\ldots ,x_s\}^*.$$
\end{enumerate}
In particular, this gives a way of testing ideal membership: to test whether $f_0\in I$, we simply check whether there is some $h_i$ such that $${\rm in}(f)\in (C\setminus \{0\}) \{x_1,\ldots ,x_s\}^* {\rm in}(h_i)\{x_1,\ldots ,x_s\}^*;$$ if not, then $f_0\not\in I$ and we may stop; if this is the case, we can find $c\in C\setminus \{0\}$ and $a,b\in \{x_1,\ldots ,x_s\}^*$ such that the monomial occurring in ${\rm in}(f_0-cah_ib)$ is degree lexicographically less than the monomial in ${\rm in}(f_0)$.  It then suffices to check that $f_1:=f-cah_ib \in I$, and by applying the procedure above we obtain a sequence $f_0,f_1,\ldots $ which must terminate since the collection of monomials is well-ordered with respect to $\prec$; in particular, we either obtain that $f_n\not\in I$ for some $n$, in which case $f_0\not\in I$; or we get $f_n=0$ for some $n$, in which case $f_0$ is in $I$.

We are now able to complete the analysis for the class $\mathcal{A}_p$. We may assume that $\bar{P}_1,\ldots ,\bar{P}_m$ are not all zero and we let $D$ denote the largest degree of a nonzero element of $\{\bar{P}_1,\ldots ,\bar{P}_k\}$ and we let $d$ denote the gcd of the coefficients occurring in the elements $\{\bar{P}_1,\ldots ,\bar{P}_m\}$.  Then each monomial that occurs with nonzero coefficient in an element from $\{\bar{P}_1,\ldots ,\bar{P}_k\}$ is of the form $X_1^{i_1}\cdots X_s^{i_s}$ with $0\le i_1,\ldots ,i_s<D+1$. Then since the elements
$$i_1+i_2(D+1)+\cdots +i_s(D+1)^{s-1}$$ with $0\le i_1,\ldots ,i_s<D+1$ are pairwise distinct, we see that the univariate polynomials
$$G_i(X):=P_i(X,X^{D+1}\ldots ,X^{(D+1)^{s-1}})=\Phi(P_i)(X,X^{D+1},\ldots ,X^{(D+1)^{s-1}})$$ have the property that the gcd of the coefficients occurring in the elements $$\{G_1(X),\ldots ,G_m(X)\}$$ is again $d$; moreover, each $G_i$ is of degree at most $(D+1)^s$.  Then there is some $i\in \{1,\ldots ,m\}$ and some $j\in \{0,\ldots , (D+1)^s\}$ such that $G_i(j)\neq 0$.  We then let $N$ denote the gcd of all elements of the form $G_i(j)$ with $i\in \{1,\ldots ,m\}$ and $j\in \{0,\ldots , (D+1)^s\}$, which we can compute.  It follows that if $P_1,\ldots ,P_m$ are identities for $R$ then $N=0$ in $R$.  In particular, if $|N|$ has prime factorization $p_1^{a_1}\cdots p_s^{a_s}$ and $R\in \mathcal{A}_p$ is such that $P_1,\ldots ,P_m$ are identities for $R$ then $p\in \{p_1,\dots ,p_s\}$ and if $p=p_i$ then $p_i^{a_i}=0$ in $R$.

So we will now give the procedure for dealing with this finite set of primes.  For $p\in \{p_1,\ldots ,p_s\}$ we then have that if $R\in \mathcal{A}_p$ satisfies the identities $$P_1=\cdots =P_m=0$$ then we have some fixed computable $a$ such that $p^a=0$ (here $a=a_i$, where $i$ is such that $p=p_i$).  Now if $p^a\mid d$ then $\bar{P}_1,\ldots ,\bar{P}_m$ all vanish on $R$ and so in the expression given in Equation (\ref{eq:Hk}), we have that each $H_k$ vanishes on $R$; in particular, in this situation, we may assume without loss of generality that each $H_k$ is identically zero and so we appeal to the earlier case we considered where this occurs.

Thus we may assume that $p^a\nmid d$, and we let $p^b=\gcd(p^a,d)$ with $b<a$.  Then there is some $k$ such that the gcd of $p^a$ and the coefficients occurring in $G_k$ is exactly $p^b$.
Then there are integer polynomials $F$ and $F'$ such that $G_k(X) = p^b F(X) + p^{b+1} F'(X)$ and such that $F(X)$ is nonzero and no nonzero coefficient in $F(X)$ is a multiple of $p$.  Then since $G_k$ is an identity for $R$, we see that
$$p^b (F(X) + p F'(X))\left(\sum_{j=0}^{a-b-1} F(X)^{a-b-1-j} (pF'(X))^{j}(-1)^j\right),$$
which is  $$p^b F(X)^{a-b} +(-1)^{a-b-1} p^a F'(X)^{a-b},$$ is also an identity for $R$. Since $p^a=0$ in $R$, we then see that
$p^b F(X)^a$ is zero on $R$.  Now let $t$ denote the smallest power of $X$ that occurs with a nonzero coefficient in $F(X)$.  By multiplying $G_k$ by a unit in $\mathbb{Z}/p^a\mathbb{Z}$, we may assume that the coefficient of $X^t$ is $1$ in $F(X)$.  Then $t\ge 1$ since otherwise we'd have $p\nmid F(0)$, but by construction $p^a$ must divide $p^b(F(0)+pF'(0))$, so this can't be the case.  Thus $t\ge 1$.  It follows that for $u\in J(R)$ we have $p^b u^{at} =0$ in $R$ since $p^b F(u)^a \in p^b u^{at}(1+J(R))$.

We now show how one can use Gr\"obner-Shirshov bases to complete the decision procedure in this remaining case.  We claim that there is a noncommutative ring $R$ in the class $\mathcal{A}_p$ for which $P_1,\ldots, P_m$ are all identities for $R$ if and only if
$$S:=\mathbb{Z}\{X,Y\}/((p^a, p^b (X,Y)^{at})+L)$$ is noncommutative, where
$L$ is the finitely generated ideal generated by all specializations of $P_1,\ldots ,P_m$ at elements of the form
$X_i=\sum a_w w$, where $w$ runs over words in $X$ and $Y$ of length less than $at$ and $a_w\in \{0,\ldots ,p^a-1\}$.  To see this, notice that since $P_1,\ldots ,P_m$ are identities for $R$ all elements of $L$ must be zero in $R$ and since $p^a=0$ in $R$, we see that $R$ is a homomorphic image of $S$.
Now it suffices to show that $P_1,\ldots ,P_m$ are identities for $S$.  Notice that if
$z_1,\ldots ,z_s\in S$, then we may write $z_i =b_i + c_i$ with $b_i$ of the form $\sum a_w w$, where $w$ runs over words in $X$ and $Y$ of length less than $at$ and $a_w\in \{0,\ldots ,p^a-1\}$, and $c_i\in (p^a, (X,Y)^{at})$.
Then by construction $$P_i(b_1,\ldots ,b_s) \in L.$$
Then using Equation (\ref{eq:Hk}), $$P_i(z_1,\ldots ,z_s)=\bar{P_i}(z_1,\ldots ,z_s) + \sum A_{i,j,k}[z_j,z_k] ~(\bmod ~I).$$
Since every coefficient of $\bar{P}_i$ is divisible by $p^b$ and since the image of $p^b(X,Y)^{at}$ is zero in $S$, we see that the images of $\bar{P_i}(z_1,\ldots ,z_s)$ and $ \bar{P}_i(b_1,\ldots ,b_s)$ are equal in $S$.
Similarly, since the images of $(p,X,Y)[X,Y]$ and $[X,Y](p,X,Y)$ are zero in $S$ we see that
the images of $\sum A_{i,j,k}(z_1,\ldots ,z_s) [z_j,z_k] $ and $\sum A_{i,j,k}(b_1,\ldots ,b_s) [b_j,b_k] $ in $S$ are equal.  It follows that $P_i(z_1,\ldots ,z_s)=P_i(b_1,\ldots ,b_s)=0$, and so we have proved the claim.

Now we put a degree lexicographic order on the monomials in $X$ and $Y$ with $Y\succ X$ and we consider the finitely generated ideal generated by $$p^a,  p^j (X,Y)^{at}, (p,X,Y)[X,Y], [X,Y](p,X,Y)$$ along with the finitely many generators 
for $L$.  Then by construction $pYX$, $XYX$, $YYX$, $YXX$, and $YXY$ are all initial terms of elements of $I$.  It 
follows that every element of our Gr\"obner-Shirshov basis whose leading term is not on the list $$\{pYX, XYX, YYX, YXX, YXY\}$$ 
cannot be a multiple of one of these terms and hence must be (after multiplying a suitable unit in $\mathbb{Z}/p^a\mathbb{Z}$) $YX$ or of the from $\alpha X^i Y^j$ with $\alpha\in \{1,p,\ldots ,p^{a-1}\}$.
We now claim that $I$ must have a finite Gr\"obner-Shirshov basis.  To see this, suppose that this is not the case.  Then there must be an infinite collection of elements $f_1,f_2,\ldots $ in our Gr\"obner-Shirshov basis whose initial terms are of the form $p^k X^i Y^j$ with $k$ a fixed element in $\{0,1,\ldots, {a-1}\}$.  Thus we have ${\rm in}(f_i)=p^k X^{a_i}Y^{b_i}$ with the pairs $(a_i,b_i)$ pairwise distinct, since the initial terms of the $f_i$ are pairwise distinct.  Moreover, for $i\neq j$, ${\rm in}(f_j)$ cannot be in the monomial ideal generated by ${\rm in}(f_i)$.  Thus we see that for $i\neq j$ we must have either $a_i< a_j$ and $b_i > b_j$ or $a_i>a_j$ and $b_i <b_j$.  Now let $a=\min\{a_i \colon i\ge 1\}$ and pick $i_0$ such that $a_{i_0}=a$.  We let $b=b_{i_0}$.  Then
${\rm in}(f_{i_0}) = p^k X^a Y^b$ and since $X^a Y^b$ is not in the monomial ideal generated by $X^{a_j} Y^{b_j}$ for $j\neq i_{0}$ and since each $a_j\ge a$, we see that $b_j \le b$ for all $j$.  A symmetric argument shows that there is some $c$ such that $a_j\le c$ for all $j$.  But there are only finitely many monomials of the form $X^e Y^f$ with $e\le c$ and $f\le b$, which contradicts the fact that the initial terms of our monomials are pairwise distinct.  Thus $I$ has a finite 
Gr\"obner-Shirshov basis, which means, in particular, that the Gr\"obner-Shirshov basis algorithm, applied to the ideal $I$, necessarily terminates.  Then since ideal membership is testable when one has a finite Gr\"obner-Shirshov basis, we can test whether $[X,Y]\in I$ and in particular, we can determine whether our ring is commutative.
\subsection{Examples of application of the algorithm}
The algorithm above makes use of the three types of rings that occur in the statement of Theorem \ref{thm:trichotomy}.  We give a few examples to show how this algorithm can be applied in practice.
\begin{example} The identity $P(X,Y)=X^2 Y^2 + X^4 Y^2 + XYXY$ forces a ring to be commutative.
\end{example}
\begin{proof} If we follow the steps, we first compute $P$ mod the commutator ideal in terms of the basis $\{X^i Y^j\}$ and we get $P(X,Y)=2X^2 Y^2 + X^4 Y^2 - X[X,Y]Y$.  Since $\bar{P}$ is nonzero, we see that taking $X=1, Y=1$ gives that $3=0$ in $R$.  In particular, if there is a noncommutative ring for which $P=0$ is an identity then $P$ must be an identity for $U_3$, a ring of the form $B_{3,n,i}$ with $n\ge 2$, or a ring in $\mathcal{A}_3$ with $3=0$. For $U_3$, we compute and find that $X=e_{1,1}$, $Y=e_{1,2}+e_{2,2}$ gives $P(X,Y) = -e_{1,2}\neq 0$ and so $P$ is not an identity for $U_3$.  For $B_{3,n,i}$ we have the additional identity $Z[X,Y]=[X,Y]Z^{3^i}$ and so if $P=0$ is an identity for $B_{3,n,i}$ then so is
$2 X^2 Y^2 + X^4 Y^2 - [X,Y]X^{3^i}Y=0$.  Since $B_{3,n,i}$ mod its commutator ideal is $\mathbb{F}_{3^n}$, we have that $ 2X^2 Y^2 + X^4 Y^2=0$ must be an identity for $\mathbb{F}_{3^n}$ and so by Lemma \ref{lem:Alon}, we must have $n\le 1$, a contradiction.  Thus $P=0$ cannot be an identity for $B_{3,n,i}$ with $n\ge 2$. Finally, if $P=0$ is an identity for a ring $R$ in $\mathcal{A}_3$ with $3=0$, then following the algorithm we find that
$P(X,X)=2X^4+X^6=-X^4(1-X^2)$ is an identity for $R$.  Then we have $x^4=0$ for all $x\in J(R)$ since $1-x^2$ is a unit whenever $x\in J(R)$.  

We now claim that $J(R)^7=(0)$.  To see this, since $R\in \mathcal{A}_3$, $R$ is generated by two elements $u,v\in J(R)$, it suffices to show that all monomials in $u$ and $v$ of length seven are equal to zero in $R$.  
Since $R\in \mathcal{A}_3$, we have $(0)=[R,R]J(R)=J(R)[R,R]$.  Now we suppose that some monomial $w$ of length seven in $u$ and $v$ is nonzero.  Then since $u^4=v^4=0$, after possibly switching the labels of $u$ and $v$, $w$ must be of the form $u^a v^b w'$ with $0<a,b\le 3$ and $w'$ a monomial of length $7-a-b\ge 1$ in $u$ and $v$ that starts with $u$.  Among all such nonzero monomials, we pick one with $a$ maximal.  Then since $u^a\in J(R)$, 
$u^a [v^b, w'] =0$ in $R$ and so $u^a w' v^b  = u^a v^b w'$. In particular, $u^a w' v^b$ is nonzero.  But $u^a w'$ is a monomial with $u^{a+1}$ as a prefix, which contradicts the maximality of $a$.  Thus we obtain the claim.

Since $[[X,Y],Z]=0$ is an identity for every ring in $\mathcal{A}_3$ we see that
$ 2X^2 Y^2 + X^4 Y^2 - [X,Y]XY=0$ is an identity in $R$.  At this point in the algorithm we would normally use Gr\"obner-Shirshov bases, which are computationally non-trivial, but in this case one can use an ad hoc argument to simplify things.
Note that $R$ is an $\mathbb{F}_3$-algebra generated by elements $u,v$ with $(u,v)^7=(0)$.  Then we first compute all evaluations of $P$ at $\mathbb{F}_3$-linear combinations of words of length at most $6$ in $u$ and $v$.  Doing this, we find that $0=P(1+u,1+v)\in  v+ J(R)^2$.  Since all elements in $J(R)^2$ are central we then get $[v,u]=0$ and so $R$ is commutative.
\end{proof}
We give a second example, illustrating the other key case: when $P$ is in the commutator ideal.
\begin{example} Let $P(X,Y)=X^2YXY -X^2Y^2X -XYX^2Y +XY^2X^2 +YX^2YX - YXYX^2$. Then there is a noncommutative ring $R$ for which $P=0$ is an identity.
\end{example}
\begin{proof} To implement the algorithm, we compute $P$ modulo the commutator ideal, using the basis $\{X^i Y^j\}$.
Doing so, we find
$$P(X,Y)= X[X,Y]XY - X[X,Y]YX -[X,Y]XYX +[X,Y]YX^2.$$
Then for any ring $R$ in the statement of Theorem \ref{thm:trichotomy} we have $[R,R]R[R,R]=(0)$ and so $P=0$ holds for $R$ if and only if the identity
$X[X,Y]XY -X[X,Y]XY -[X,Y]X^2Y -[X,Y]X^2Y = 0$ holds for $R$.  But since this is the identity $0=0$, we see that in fact $P$ is an identity for every ring whose commutator ideal has square zero.  In particular, $P=0$ is an identity for the ring $U_2$.
\end{proof}

The next example is a special case of the main result from \cite{AD}, which is more general in that it allows $m$ and $n$ to depend on the elements of the ring and also allows there to be a sign that depends on the elements of the ring.

\begin{example} If $m,n\ge 2$ and have opposite parity and $P(X)=X^m-X^n$ then the identity $P(X)=0$ forces a ring to be commutative.
\end{example}
\begin{proof} To apply the algorithm in practice one must have fixed $m$ and $n$, but we shall use ad hoc arguments to get around this restriction. First, we find $P(-1)=0$ gives that $2=0$ in every ring for which $P(X)=0$ is an identity. Thus if there is a noncommutative ring for which $P(X)=0$ is an identity then it must hold for either $U_2$, a ring of the form $B_{2,n,i}$ with $n\ge 2$, or a ring from $\mathcal{A}_2$ with $2=0$.
The case of $U_2$ is straightforward: take $X= e_{1,1}+e_{1,2}+e_{2,2}$; then $X^k=1$ when $k$ is even and $X^k = X$ when $k$ is odd, so $X^m\neq X^n$ in $U_2$.
To actually apply the algorithm for $B_{2,n,i}$, we would bound $n$ in terms of the degree of $P$ for a fixed $m$ and $n$ and then check whether any of the resulting finite set of rings have $P=0$ as an identity. But we observe in this case we can again take $X= e_{1,1}+e_{1,2}+e_{2,2}$ and we get that $X^m- X^n$ is a non-identity for $B_{2,n,i}$.  Finally, if $R$ is in $\mathcal{A}_2$ and we take $X=1+u$ with $u\in J(R)\setminus J(R)^2$ then
$X^k \in 1+J(R)^2$ if $k$ is even and $X^k \in 1+u+J(R)^2$ if $k$ is odd. In particular, if $X^m= X^n$ we see that $u\in J(R)^2$, which is a contradiction. Thus the identity $P(X)=0$ forces a ring to be commutative.
\end{proof}

\section{Jacobson's and Herstein's theorems revisited}
\label{Applications}
A famous theorem of Jacobson \cite[Theorem 11]{J} asserts that if a ring $R$ has the property that, for each $x\in R$, there exists an integer $n(x)>1$ depending on $x$ with $x^{n(x)}=x$ (such rings are further called {\it potent}), then $R$ is commutative. On the other hand, Herstein \cite{H} generalized this important assertion by proving that if $R$ is a ring with center $Z(R)$ such that $x^{n(x)}-x\in Z(R)$ for every $x\in R$, then $R$ is necessarily commutative. A recent generalization of Jacobson's result is given in \cite{AD}: If $R$ is a ring such that, for any $x\in R$, there are two integers $n(x)>m(x)>1$ of opposite parity with $x^{n(x)}=x^{m(x)}$, then $R$ is commutative.

So, taking into account the statements alluded to above, it is quite logical to consider those rings $R$ for which $x^{n(x)}-x^{m(x)}\in Z(R)$. However, the next construction illustrates that situation is more complicated than one might anticipate. In fact, if $p\geq 3$ and we take $R=\mathbb{F}_p\{x,y\}/J$, where $J$ is the ideal $(x,y)^3$ (the cube of the homogeneous maximal ideal), then every element $a$ in $R$ can be written as $c+u$ with $c$ in $\mathbb{F}_p$ and $u$ in the homogeneous maximal ideal. Since $u^3=0$ and $p\geq 3$, we see that $(c+u)^p = c^p + u^p = c$ and so $a^{2p} - a^p$ is a central element for all $a$ in $R$. Notice $R$ is not commutative since $xy-yx$ is not in $J$ by construction. Thus additional conditions are required to obtain a commutativity theorem.

\medskip

The following result gives some further advantage in discovering the discrepancies in commutativity of rings when we involve the center $Z(R)$.

\begin{thm} Let $P(X)\in \mathbb{Z}[X]$ be a polynomial. Then the following statements hold:
\begin{enumerate}
\item there is a noncommutative ring $R$ for which $P(X)=0$ is an identity if and only if there is a prime $p$ such that $P(X)\in (p,(X^p-X)^2)\mathbb{Z}[X]$;
\item  there is a noncommutative ring $R$ for which $P(X)Y=YP(X)$ is an identity if and only if there is a prime $p$ such that the first derivative, $P'(X)$, of $P(X)$ is in the ideal $$(p,X^p-X)\mathbb{Z}[X].$$
\end{enumerate}
\label{thm:Herstein}
\end{thm}
\begin{proof}
If $P(X)\in (p,(X^p-X)^2)$, then $P(X)=0$ is an identity for the noncommutative ring $U_p$ and if $P'(X)\in (p,X^p-X)$, then $[P(X),Y]=0$ is an identity for the (noncommutative) ring
$\mathbb{F}_p\{u,v\}/(u,v)^3$.    Thus one direction is immediate.

We now consider the more difficult direction. Suppose that there is a ring $R$ that is not commutative for which $P(X)=0$ is an identity. Then we may assume that $R$ lands in one of the cases (a)--(c) given in the statement of Theorem~\ref{thm:trichotomy}. In particular, there is some prime $p$ and some $m\ge 1$ such that $R$ has characteristic $p^m$. In all cases, $R/J(R)$ contains a subring isomorphic to $\mathbb{F}_p$ and so $X^p-X$ must divide $P(X)$ mod $p$, since $P(X)$ must be an identity for $\mathbb{F}_p$. Thus we may write $P(X)=(X^p-X)Q(X)$ modulo $p$.  Now let $u=\alpha+[s,t]$ with $\alpha\in \mathbb{Z}$ and $[s,t]$ a nonzero commutator in $R$. We now consider two cases. The first case is when $R$ is an $\mathbb{F}_p$-algebra.
Then since $[R,R]^2=p[R,R]=0$, $0=P(u) =  ((\alpha^p-\alpha) - [s,t])Q(u)$.  Moreover, since $\alpha^p-\alpha=0$ in $R$, we have
$$P(u)= ((\alpha^p-\alpha) - [s,t])Q(u)= - [s,t]Q(\alpha).$$  So we must have $Q(\alpha)=0$ in $R$ for every $\alpha\in \mathbb{Z}$, which implies that $$X^p-X\mid Q(X)~(\bmod~p),$$ and so we are done in this case.  Alternatively, $R$ has characteristic $p^m$ with $m$ strictly larger than $1$ and so for $v\in R$ we have
$$0=P(v+p^{m-1}) = P(v) + p^{m-1} P'(v)=p^{m-1}P'(v).$$  It follows that $p^{m-1}P'(X)$ is also an identity for $R$ and since $R$ has characteristic $p^m$, we see that $p\mid P'(\alpha)$ for every $\alpha\in \mathbb{Z}$ and so 
$P(X)\in (p,(X^p-X)^2)$ as required.

Next suppose that $P(X)\in \mathbb{Z}[X]$ is a polynomial and that there exists a ring $R$ which is not commutative such that $P(X)Y-YP(X)=0$ is an identity for $R$.  Then again we may assume that $R$ is covered by one of the cases (a)--(c) given in the statement of Theorem \ref{thm:trichotomy}, and so there is a prime $p$ and $m\ge 1$ such that $p^mR=(0)$. We first consider the case when not all commutators are central. Then $R$ is not in the class $\mathcal{A}_p$ and so it is an $\mathbb{F}_p$-algebra.  Then there exist $u,v$ and $z$ in $R$ such that $[[u,v],z]\neq 0$. But since $[u,v]^2=0$, one has for $\alpha\in \mathbb{F}_p$ that $P(\alpha +[u,v]) = P(\alpha)+P'(\alpha)[u,v]$ and so since $[P(\alpha+[u,v]),z]=0$, we see that $P'(\alpha)=0$ for all $\alpha\in \mathbb{F}_p$.  In particular, $X^p-X$ divides $P'(X)$ mod $p$, and so we obtain the result in this case.

We now consider the case when all commutators are central in $R$ and so $R$ is in the class $\mathcal{A}_p$. Given ideals $I$ and $J$ of $R$, we let $[I,J]$ denote the two sided ideal of $R$ generated by commutators $[u,v]$ with $u\in I$ and $v\in J$. Then, in this case,
$[[R,R],R]=(0)$ and $[R,R]^2=(0)$. For $x,y\in R$, $$P(x)y-yP(x)\equiv P'(x)[x,y]~(\bmod ~L),$$ where $L=[[R,R],R]$.  Since the annihilator of nonzero commutators contains the Jacobson radical and since $R/J(R)\cong \mathbb{F}_p$, we therefore see that $P'(X)$ is an identity for a finite field of characteristic $p$. In particular, $P'(X)\in (p,X^p-X)\mathbb{Z}[X]$, and so we get the result in this case too.
\end{proof}
We point out that Theorem \ref{thm:algorithm} shows that one can decide whether there is some prime $p$ for which either (1) or (2) holds in the statement of Theorem \ref{thm:Herstein}.  In this case, however, the decision procedure can be performed much more quickly.  For example, for (1), if $P(X)$ has degree $d$ and is nonzero, one can find some $i\in \{0,\ldots ,d\}$ such that $N:=P(i)$ is nonzero.  If $P(X)$ is an identity for a ring $R$ then $N=0$ in $R$ and so if there exists a prime $p$ for which $P(X)\in (p,(X^p-X)^2)$ then $p\mid N$.  For the finite set of primes $p$ for which $p\mid N$ one can then check whether $P(X)$ is in this ideal.  The condition (2) can be checked similarly.

As a consequence, we obtain a general commutativity theorem.

\begin{cor} Let $a$ and $b$ be positive integers with $a>b$.  A ring satisfying the identity $$[X^a-X^b,Y]=0$$ is necessarily commutative if and only if one of the following conditions hold:
\begin{enumerate}
\item $b=1$;
\item $\gcd(a,b)=1$ and $a$ and $b$ have opposite parity.
\end{enumerate}
\end{cor}
\label{cor:XXX}
\begin{proof}
Let $P(X)=X^a-X^b$. It suffices to show that there is a prime $p$ such that $P'(X)\in (p, X^p-X)$ if and only if $b>1$ and either $\gcd(a,b)>1$ or $a$ and $b$ have the same parity.

Notice $P'(X)=  X^{b-1}(aX^{a-b}-b)$ and so if $b=1$, then $P'(X)=aX^{a-b}-1$, which is nonzero mod $p$ when $X=0$ and hence there is no prime $p$ such that $P'(X)\not \in (p,X^p-X)$ in this case. Next suppose that $\gcd(a,b)=1$ and $a$ and $b$ have opposite parity. In particular, there is no prime $p$ such that $p^2-p$ divides $a-b$, since $p^2-p$ is always even.  If there is some prime $p$ such that $P'(X)\in (p,X^p-X)$, then since $P'(1)=a-b$, $p|a-b$ and so 
$$P'(X)\equiv a X^{b-1}(X^{a-b}-1)~(\bmod~p).$$ 
Since $\gcd(a,b)=1$, we then see $p\nmid a$ and thus if $P'(X)\in (p,X^p-X)$ then $X^{p-1}-1$ must divide $X^{a-b}-1$ mod $p$, and so $(p-1)\mid a-b$. Thus $p(p-1)\mid a-b$, a contradiction, since they have opposite parity. Then, by Theorem \ref{thm:Herstein}, a ring satisfying the identity $[X^a-X^b,Y]=0$ with $a$ and $b$ satisfying the above conditions is necessarily commutative. 

To see the other direction, suppose that $b>1$ and either $\gcd(a,b)>1$ or $2$ divides $a-b$. If there is some prime $q$ such that $q|\gcd(a,b)$, then $P'(X)\equiv 0~(\bmod~q)$, and so Theorem \ref{thm:Herstein} (2) then shows that the condition $\gcd(a,b)=1$ is necessary. If $b>1$ and $2$ divides $a-b$, then
$P'(X)\equiv a X^{b-1}(X^{a-b}-1)~(\bmod~2)$. Since $b>1$, it follows that $P'(0)\equiv 0~(\bmod ~q)$; and since $X-1$ divides $P'(X)$ mod $2$, it must be that $P'(X)\in (2,X^2-X)$, and so we see necessity of the condition that $a$ and $b$ have opposite parity from Theorem \ref{thm:Herstein} (2).
\end{proof}

Notice one can rephrase Corollary \ref{cor:XXX} as follows: \emph{Let $R$ is a ring and $n$, $m$ are fixed integers greater than $1$ of opposite parity such that $\gcd(m,n)=1$. If, for all $x\in R$, $x^n-x^m\in Z(R)$, then $R$ is necessarily commutative.}  We point out that this somewhat extends the mentioned above results from \cite{H} and \cite{AD} for fixed degrees.  The general form of the results from \cite{H} and \cite{AD} suggest the following should hold. 
\begin{conj*} Suppose that $R$ is a ring such that for every $x\in R$ there exist positive integers $a=a(x)$ and $b=b(x)$, depending on $x$, such that:
\begin{enumerate}
\item $x^{a}-x^{b}$ is central;
\item  either $b=1$ or $\gcd(a,b)=1$ and $a$ and $b$ have opposite parity.
\end{enumerate}
Then $R$ is commutative.
\end{conj*}
We give one last application of the algorithm described in \S\ref{Algorithm}.  Herstein \cite{Her2} considered rings $R$ for which the identity $(XY)^n=X^n Y^n$ holds for some $n\ge 2$.  In this case, he showed that the commutator ideal is necessarily nilpotent.
\begin{thm} Let $S\subseteq \mathbb{N}$.  Then there is a noncommutative ring satisfying the identities $(XY)^n=X^n Y^n$ for every $n\in S$ if and only if there exists a prime $p$ such that $p\mid {n\choose 2}$ for every $n\in S$.
\end{thm}
\begin{proof}
Suppose first that there is a prime $p$ such that ${n\choose 2}$ is a multiple of $p$ for every $n\in S$.  Consider the noncommutative ring $R:=\mathbb{F}_p\{X,Y\}/(X,Y)^3$. Then there is a homomorphism $\phi :R \to \mathbb{F}_p$ such that for each $a\in R$ we have $a=\alpha(a)+j(a)$, where $j(a)$ is in the Jacobson radical of $R$.  In particular, since $J(R)^3=(0)$, we have $a^p =\alpha(a)^p=\alpha(a)$ when $p\ge 3$ and we have $a^4=\alpha(a)$ when $p=2$.  If $n\in S$ then by assumption ${n\choose 2}$ is a multiple of $p$.  Thus $p|n$ or $p|(n+1)$ when $p$ is odd and $n\equiv 0,1~(\bmod~ 4)$ when $p=2$.  In either case, we see that
$a^n b^n =(ab)^n$ since the left- and right-hand sides are both either $\alpha(ab)^{n/p}$ or $\alpha(ab)^{(n-1)/p} ab$, depending on whether $n$ is a multiple of $p$ or is $1$ mod $p$.

Next suppose that for every prime $p$ there is some $n\in S$ such that ${n\choose 2}$ is not a multiple of $p$ and suppose to the contrary that there is a noncommutative ring $R$ with the property that $(ab)^n =a^n b^n$ for all $a,b\in R$ and all $n\in S$.  Then there is some prime $p$ such that the identities $(XY)^n =X^n Y^n$, for $n\in S$, hold in one of the rings given in the statement of Theorem \ref{thm:trichotomy}.  Notice that such an identity cannot hold in $U_p$, since the elements $a=e_{1,1}$ and $b=e_{1,2}+e_{2,2}$ are both idempotent and so $a^n b^n = ab = e_{1,2}$ while $(ab)^n = 0$ for $n\ge 2$.
We next consider the rings $B_{p,n,i}$.  Let $m=(p^n-1)/\gcd(p^n-1,p^i-1)$.  Then $m\ge (p^n-1)/(p^{n/2}-1)\ge p+1$ and so $m$ is either a multiple of an odd prime or a multiple of $4$.  Then by assumption there is some $n$ such that $n\not\equiv 0,1~(\bmod ~m)$.  We pick such an $n\in S$ and write $n=m n_0+b$ with $b\in \{2,\ldots ,m-1\}$.  For $\lambda\in \mathbb{F}_q$ we let
\[a_{\lambda}=\left( \begin{array}{cc} \lambda^{p^i} & 0 \\ 0 & \lambda\end{array}\right) \] and
\[b_{\lambda}=\left( \begin{array}{cc} \lambda^{p^i} & 1 \\ 0 & \lambda\end{array}\right). \]
Then by assumption
\begin{equation}
\label{eq:ablam}
a_{\lambda}^n b_{\lambda}^n =(a_{\lambda}b_{\lambda})^n.
\end{equation}
Then for $\lambda$ such that $\lambda^{p^i}\neq \lambda$, computing the $(1,2)$-entry of both sides of Equation (\ref{eq:ablam}) gives
$$\lambda^{p^i n}\cdot \left(\frac{\lambda^{p^i n} - \lambda^n}{\lambda^{p^i}-\lambda}\right) = \lambda^{p^i}\cdot  \left(\frac{\lambda^{2p^i n} - \lambda^{2n}}{\lambda^{2p^i}-\lambda^2}\right).$$
Then a simple computation shows that this holds only when either $\lambda^{(p^i-1) n}=1$ or
$1 = \lambda^{(p^i-1) (n-1)}$.  Since $n\not\equiv 0,1~(\bmod ~m)$, we see there is some $\lambda\in \mathbb{F}_q$ with $\lambda^{p^i}\neq \lambda$ such that $\lambda^{(p^i-1) n}\neq 1$ and
$\lambda^{(p^i-1) (n-1)}\neq 1$, contradicting the fact that the identity $(XY)^n =X^n Y^n$ holds in $B_{p,n,i}$.

Finally, we consider rings in the class $\mathcal{A}_p$. So suppose that $(XY)^n=X^n Y^n$, for each $n\in S$, is an identity for a ring $R$ in $\mathcal{A}_p$ for some prime $p$.  Then by assumption there is some $n\in S$ such that $p\nmid {n\choose 2}$.  Then there is some smallest $k\ge 1$ such that $p^k =0$ in $R$.  Let $\alpha: R\to \mathbb{Z}/p^k \mathbb{Z}$ be the surjection obtained by reducing modulo the nilpotent radical.  Then if $a\in R$ we have $a=\alpha(a)+x(a)$ for some $x(a)\in J(R)$.  Moreover, there is some fixed $m$ such that $x^{p^m}=0$ for every $x\in J(R)$ and $m\ge 2$ if $p=2$.  Since $n\not\equiv 0,1~(\bmod ~p^m)$, we can write $n=p^m n_0+b$ with $b\in \{2,\ldots ,p^m-1\}$.  Then for $u\in R$ we have $u^n = u^{p^m n_0+b} = \alpha(u)^{p^m n_0} u^b$.  Thus if $G$ is the subgroup of $R^*$ consisting of elements of the form $1+x$ with $x\in J(R)$, we have $x^b y^b = (xy)^b$ for all $x,y\in G$.  Notice, however, that $G$ is a finite $p$-group and hence is nilpotent.  Moreover, $G$ generates $R$ as a $\mathbb{Z}$-algebra and since $R$ is noncommutative, $G$ must be a nonabelian nilpotent group.  In particular, $G$ has a normal subgroup $N$ such that $H:=G/N$ has the property that $H/Z(H)$ is abelian, where $Z(H)$ is the centre of $H$.  It follows that there exist $x,y\in H$ such that $yx=zxy$ with $z\in Z(H)$.  Since $x^b y^b=(xy)^b$ in $G$ this holds in $H$ and so $x^b y^b = (xy)^b = z^{b\choose 2} x^b y^b$.  Since $H$ is a $p$-group, ${b\choose 2}$ must be a multiple of $p$.  Since ${n\choose 2}\equiv {b\choose 2}~(\bmod ~p)$, we see that $p\mid {n\choose 2}$, a contradiction.
\end{proof}
\begin{remark} Herstein \cite{Her2} also considered the case of identities of the form $$(X+Y)^n=X^n+Y^n.$$ Here the answer is simpler for which sets $S$ of natural numbers allow there to be a noncommutative ring such that $(X+Y)^n=X^n+Y^n$ for every $n\in S$.  This can only be the case if there is a prime $p$ such that $S\subseteq T_p:=\{p,p^2,p^3,\ldots \}$ when $p$ is odd; and $S\subseteq T_2:=\{4,8,\ldots \}$.  To see this, observe that these identities hold for all $n\in T_p$ for the ring $\mathbb{F}_p\{X,Y\}/(X,Y)^3$.  On the other hand, if $(X+Y)^n=X^n+Y^n$ holds for a ring occurring in the statement of Theorem \ref{thm:trichotomy}, it must also hold on $\mathbb{F}_p$ for some prime $p$, since this is always a subring of a homomorphic image of these rings.  But then it is easily checked that this forces $n$ to be a power of this fixed prime.  In the special case $p=2$, the identity $(X+Y)^2=X^2+Y^2$ forces a ring to be commutative.
\end{remark}
\section{Multilinear identities that force commutativity}
\label{sec:multilinear}
In this brief section we give a proof of a general result that in particular implies Theorem \ref{thm:multi} when we take $\mathcal{S}$ below to comprise a single identity.

\begin{thm}
Let $\mathcal{S}$ be a set of homogeneous multilinear polynomials with integer coefficients. Then there is a noncommutative ring $R$ for which every element of $\mathcal{S}$ is an identity if and only if there is some fixed prime $p$ such that whenever $P(X_1,\ldots ,X_m)=\sum_{\sigma \in S_m} c_{\sigma}X_{\sigma(1)}\cdots X_{\sigma(m)} \in \mathbb{Z}\{X_1,\ldots ,X_m\}$ is an element of $\mathcal{S}$
the following hold:
\begin{enumerate}
\item $p\mid P(1,1,\ldots ,1)$;
\item $p\mid \Theta_{i,j}(P)$ for $1\le i<j\le m$,
\end{enumerate}
where the $\Theta_{i,j}(P)$ are as defined in Equation (\ref{eq:Theta}).
 Moreover, if there is such a prime $p$ for which these conditions hold, then every element of $\mathcal{S}$ is an identity for the noncommutative ring $\mathbb{F}_p\{U,V\}/(U^2,V^2,UV)$.  
 \label{thm:multi2}
\end{thm}
\vskip 2mm
\begin{proof} If the elements of $\mathcal{S}$ are identities for a noncommutative ring $R$ then by Theorem \ref{thm:trichotomy} there is a prime $p$ such that there is a ring $R$ from one of the three classes of rings associated to the prime $p$ from the statement of the theorem for which each element of $\mathcal{S}$ is an identity.  Then if $P(X_1,\ldots ,X_m)\in \mathcal{S}$, $P(1,1,\ldots ,1)=0$ in $R$ and so $p\mid P(1,1,\ldots ,1)$.  Notice that if $i<j$ and we specialize $P(X_1,\ldots, X_s)$ taking $X_k=1$ for $k\neq \{i,j\}$ and $X_i=r$ and $X_j=s$ with $r,s\in R$ such that $[r,s]\neq 0$ then $P(X_1,\ldots ,X_s)$ becomes $\Theta_{i,j}(P)rs + (P(1,1,\ldots ,1)-\Theta_{i,j}(P))sr$.  Then since $P(1,1,\ldots ,1)=0$ in $R$ and $rs-sr\neq 0$ we see that $\Theta_{i,j}(P)$ annihilates $rs-sr$ and hence it must have non-trivial gcd with the characteristic of $R$, which is a power of $p$.  It follows that $p\mid \Theta_{i,j}(R)$ whenever $i<j$. Thus we see the necessity of these conditions.

Now suppose that there exists a prime $p$ such that whenever $P(X_1,\ldots ,X_m)\in \mathcal{S}$, we have $p\mid P(1,1,\ldots ,1)$ and $p\mid \Theta_{i,j}(P)$ whenever $i<j$.  Consider the ring $S:=\mathbb{F}_p\{U,V\}/(U^2,V^2,UV)$ and let $u$ and $v$ denote the images of $U$ and $V$ in $S$ respectively. Then $S$ is a noncommutative $4$-dimensional $\mathbb{F}_p$-algebra with basis $\mathcal{T}=\{1,u,v,vu\}$. We claim that $P$ is an identity for $S$.  To see this, since $P$ is multilinear, it suffices to show that $P$ vanishes whenever it is evaluated at $s$-tuples in $\mathcal{T}^s$.  Moreover, if $z_1,\ldots ,z_m$ in $S$ commute then $P(z_1,\ldots ,z_m)=P(1,1,\ldots ,1)z_1\cdots z_m =0$ and since $vu$ and $1$ are central in $S$, we then see that it suffices to consider $s$-tuples in $\mathcal{T}^s$ with at least one copy of $u$ and at least one copy of $v$.  Moreover, since $(u,v)^3=(0)$, we now see it suffices to consider $s$-tuples with exactly one copy of $u$, exactly one copy of $v$, and all other elements equal to $1$.  If we take $i<j$ and $X_i=u$, $X_j=v$, and $X_k=1$ for $k\neq i,j$ then 
when we specialize $P$ at these values of $X_1,\ldots ,X_m$ we obtain $\Theta_{i,j}(P)[u,v] =0$, since $p\mid \Theta_{i,j}(P)$.  Thus $P$ vanishes at all $s$-tuples in $\mathcal{T}^s$ and so it is a polynomial identity for $S$. The result follows.
\end{proof}

To give an example of how to apply Theorem \ref{thm:multi2}, observe that if $m=3$ and $$P(X_1,X_2,X_3) = \sum_{\sigma\in S_3} X_{\sigma(1)}X_{\sigma(2)}X_{\sigma(3)},$$ then $P(1,1,1)=6$ and $\Theta_{i,j}(P) =3$ for $1\le i<j\le 3$.  Then we see that the conditions in the statement of the theorem are satisfied with the prime $p=3$ and $P=0$ is an identity for the ring $\mathbb{F}_3\{U,V\}/(U^2,V^2,UV)$, which has $81$ elements.  
\vskip 2mm
 In light of Theorem \ref{thm:multi2}, it is natural to ask for a finite set of identities that generate the identities for the ring $\mathbb{F}_p\{U,V\}/(U^2,V^2, UV)$, since these are the identities that do not help in the context of proving commutativity for multilinear identities.  
 
Notice that 
\begin{equation}\label{eq:idents}
\begin{array}{ll}
(Z_1^p-Z_1)Z_2(Z_3^p-Z_3)Z_4(Z_5^p-Z_5)=0 & [[Z_1,Z_2],Z_3]=0 \\
(Z_1^p-Z_1)Z_2 [Z_3,Z_4]=0 & [Z_1,Z_2]Z_3[Z_4,Z_5]=0 \\
{[}Z_1,Z_2] Z_3(Z_4^p-Z_4)=0 & p=0.
\end{array}
\end{equation}
are polynomial identities for
$\mathbb{F}_p\{U,V\}/(U^2,V^2,UV)$.  We point out this is not a minimal set: for example, one can deduce 
the identity  $[Z_1,Z_2]Z_3 (Z_4^p -Z_4)=0$ from the identities 
$[[Z_1,Z_2],Z_3]=0$ and $(Z_1^p-Z_1)Z_2 [Z_3,Z_4]=0$.  We choose, however, to work with this set of identities, as it is convenient to work with.  
The following result shows that these identities generate the identities for $\mathbb{F}_p\{U,V\}/(U^2,V^2,UV)$. 
\begin{prop} Let $p$ be a prime number.  Then every polynomial identity for $\mathbb{F}_p\{U,V\}/(U^2,V^2,UV)$ is generated by the identities given in Equation (\ref{eq:idents}).
\label{prop:gen}
\end{prop}
\begin{proof} 
We make use of the notation from Section \ref{sec:dec} in this proof.
Let $$P(X_1,\ldots ,X_s)\in \mathbb{Z}\{X_1,\ldots ,X_s\}$$ be an identity for $R:=\mathbb{F}_p\{U,V\}/(U^2,V^2,UV)$ and let $u$ and $v$ denote the images of $U$ and $V$ respectively in $R$.  
Since we have the identity $p=0$ at our disposal, we may work over $\mathbb{F}_p\{X_1,\ldots ,X_s\}$ instead, so we assume now that $P\in \mathbb{F}_p\{X_1,\ldots ,X_s\}$ and we adjust $\mathcal{C}_s$ and the maps $\Phi$ from items (\ref{eq:C}) and (\ref{eq:Phi}) to reflect that our base is now $\mathbb{F}_p$.  Given two identities $P_1$ and $P_2$, we'll write $P_1\equiv P_2$ if the identity $P_1-P_2$ is implied by the identities in Equation (\ref{eq:idents}).

Our goal is to show that the identity $P=0$ is implied by the identities in Equation (\ref{eq:idents}).  
Then since $R/([R,R])\cong \mathbb{F}_p$, 
$\Phi(P)$ is an identity for $\mathbb{F}_p$. Hence by Lemma \ref{lem:Alon}, 
$\Phi(P)\in (X_1^p-X_1,\ldots ,X_s^p-X_s)\mathbb{F}_p[X_1,\ldots ,X_s]$.  
So we have
$$\Phi(P) = (X_1^p-X_1)A_1+\cdots + (X_s^p-X_s)$$ for some $A_1,\ldots, A_s\in \mathbb{F}_p[X_1,\ldots ,X_s]$.
Then notice that since $P$ is an identity for $R$, $\Phi(P)$ must be an identity for every commutative subring of $R$.
Thus $\Phi(P)$ is an identity for $\mathbb{F}_p[vu]=\mathbb{F}_p\oplus \mathbb{F}_p\cdot vu$.
Then if we take $(\lambda_1,\ldots ,\lambda_s)\in \mathbb{F}_p^s$ and we specialize $X_j = \lambda_j$ for $j\neq i$ and
$X_i=\lambda_i+vu$, then $X_j^p-X_j$ becomes zero for $j\neq i$ and $X_i^p-X_i$ becomes $-vu$.  Then the fact that $\Phi(P)$ is an identity for $\mathbb{F}_p[uv]$ gives
$$-vu A_i(\lambda_1,\ldots ,\lambda_s)=0$$ for every $(\lambda_1,\ldots ,\lambda_s)\in \mathbb{F}_p^s$.  Hence
$A_1,\ldots ,A_s$ are also identities for $\mathbb{F}_p$.  Thus we can in fact write $\Phi(P)$ as
$$\sum_{1\le i\le j\le s} (X_i^p-X_i)(X_j^p-X_j)Q_{i,j}$$ with the $Q_{i,j}\in  \mathbb{F}_p[X_1,\ldots ,X_s]$.
Since $\Phi(P-\bar{P})=0$, we then have that there are $\hat{Q}_{i,j}\in \mathcal{C}_s$ for $1\le i\le j\le s$ with 
$\Phi(\hat{Q}_{i,j})=Q_{i,j}$ and so 
$$P-\sum_{1\le i\le j\le s} (X_i^p-X_i)(X_j^p-X_j)\hat{Q}_{i,j} $$ is in the commutator ideal of 
$\mathbb{F}_p\{X_1,\ldots ,X_s\}$.  Hence
\begin{equation}
P=\sum_{1\le i\le j\le s} (X_i^p-X_i)(X_j^p-X_j)\hat{Q}_{i,j} + \sum_{1\le i<j\le s} \sum_{k=1}^{m_{i,j}} 
B_{i,j,k} [X_i,X_j] C_{i,j,k},
\end{equation} for some integers $m_{i,j}$ and polynomials
$B_{i,j,k},C_{i,j,k}\in \mathbb{Z}\{X_1,\ldots, X_s\}$ for $i$ and $j$ with $1\le i<j\le s$ and $k=1,\ldots ,m_{i,j}$.
Since $[[Z_1,Z_2], Z_3]=0$ is an identity in Equation (\ref{eq:idents}), we can reduce our expression for $P$ modulo these identities and we have
$$\sum_{k=1}^{m_{i,j}} B_{i,j,k} [X_i,X_j] C_{i,j,k} \equiv [X_i,X_j] \left(\sum_{k=1}^{m_{i,j}} B_{i,j,k} C_{i,j,k}\right).$$
 
Thus if we let $$D_{i,j}=\sum_{k=1}^{m_{i,j}} B_{i,j,k} C_{i,j,k}$$ for $1\le i<j\le s$, we see we may assume that our identity $P$ is of the form
\begin{equation}
\label{eq:newP}
\sum_{1\le i\le j\le s} (X_i^p-X_i)(X_j^p-X_j)\hat{Q}_{i,j} + \sum_{1\le i<j\le s} [X_1,X_j] D_{i,j}.
\end{equation}
Moreover, since $[Z_1,Z_2]Z_3[Z_4,Z_5]=0$ is an identity in Equation (\ref{eq:idents}), we may again work modulo the equivalence above and assume without loss of generality that each $D_{i,j}\in \mathcal{C}_s$.

We now fix $i$ and $j$ with $i<j$.  We let $(\lambda_1,\ldots ,\lambda_s)\in \mathbb{F}_p$ and we specialize our variables with
$X_k = \lambda_k$ for $k\neq i,j$, $X_i=\lambda_i+u$, $X_j=\lambda_j+v$.  Then for $k\le \ell$,
$(X_k^p-X_k)(X_{\ell}^p-X_{\ell})$ becomes zero unless $(k,\ell)=(i,j)$, but in this case it becomes $-uv$, which is also zero.  On the other hand, $[X_k,X_{\ell}]$ becomes zero under this specialization unless $(k,\ell)=(i,j)$ and
$[X_i,X_j]$ becomes $-vu\neq 0$.  Then since $(u,v)^3=(0)$, we see that under this specialization 
Equation (\ref{eq:newP}) becomes
$-vu D_{i,j}(\lambda_1,\ldots ,\lambda_s)$, and so $D_{i,j}$ is an identity for $\mathbb{F}_p$. Then Lemma \ref{lem:Alon}, gives that $D_{i,j}$ is in the ideal generated by $X_k^p-X_k$ for $1\le k\le s$ along with the commutators $[X_k,X_{\ell}]$ for $1\le k<\ell\le s$.  But this means that 
$[X_i,X_j]D_{i,j}$ is an identity for $R$ and that it is implied by the identities $[Z_1,Z_2]Z_3[Z_4,Z_5]=0$ and $[Z_1,Z_2]Z_3 (Z_4^p-Z_4)=0$ given in Equation (\ref{eq:idents}) for $1\le i<j\le s$.  Thus we can further reduce modulo our equivalence and assume that $P$ is of the form
$$\sum_{1\le i\le j\le s} (X_i^p-X_i)(X_j^p-X_j)\hat{Q}_{i,j}.$$
Now we fix $i$ and $j$ with $1\le i<j\le s$.  For $(\lambda_1,\ldots ,\lambda_s)\in \mathbb{F}_p$, we specialize 
$X_k=\lambda_k$ for $k\neq i,j$ and $X_i=\lambda_i+v$, $X_j=\lambda_j+u$.
Then under this specialization $(X_i^p-X_i)(X_j^p-X_j)$ becomes $vu\neq 0$ but 
$(X_k^p-X_k)(X_{\ell}^p-X_{\ell})$ becomes zero for $1\le k\le \ell\le s$ and $(k,\ell)\neq (i,j)$ (the case when $k=\ell=i$ and $k=\ell=j$ follow from the fact that both $u^2$ and $v^2$ are zero in $R$).  

Thus 
$$\sum_{1\le i\le j\le s} (X_i^p-X_i)(X_j^p-X_j)\hat{Q}_{i,j}$$ specializes to
$vu\hat{Q}_{i,j}(\lambda_1,\ldots ,\lambda_s)$, and so $\hat{Q}_{i,j}$ is an identity for $\mathbb{F}_p$ for $i<j$. Thus Lemma \ref{lem:Alon} gives that $\hat{Q}_{i,j}$ is in the ideal $X_k^p-X_k$ for $1\le k\le s$ along with the commutators $[X_k,X_{\ell}]$ for $1\le k<\ell\le s$.  Again, since
$$(Z_1^p-Z_1)Z_2(Z_3^p-Z_3)Z_4(Z_5^p-Z_5)=0 ~~{\rm and}~~(Z_1^p-Z_1)Z_2[Z_3,Z_4]=0$$ are identities in Equation (\ref{eq:idents}), we see that $ (X_i^p-X_i)(X_j^p-X_j)\hat{Q}_{i,j}$ is an identity for $R$ and that it is implied by the identities in Equation (\ref{eq:idents}) for $i<j$. Thus we may further reduce our identity modulo the equivalence above and assume that our identity is of the form
$$\sum_{i=1}^s (X_i^p-X_i)(X_i^p-X_i)\hat{Q}_{i,i}.$$

Now we fix $i\in \{1,\ldots ,s\}$ and for $(\lambda_1,\ldots ,\lambda_s)\in \mathbb{F}_p$, we specialize 
$X_k=\lambda_k$ and $X_i=\lambda_i +u+v$.  Then $X_k^p-X_k$ becomes zero under this specialization for $k\neq i$ and $(X_i^p-X_i)^2$ becomes $vu$.  Thus the same argument as above shows that 
$\hat{Q}_{i,i}$ is an identity for $\mathbb{F}_p$ and that $ (X_i^p-X_i)(X_i^p-X_i)\hat{Q}_{i,i}$ is implied by the identities in Equation (\ref{eq:idents}) for $i=1,\ldots ,s$.  Thus $P\equiv 0$ and the result now follows.

\end{proof}

We point out that, beyond traditional polynomial identities, there is a large body of work dealing with \emph{functional identities}, which are more general and have been developed by Bre\v{s}ar and others (see, for example, \cite{Bres1, Bres2}).  Many natural classes of functional identities yield commutativity theorems---for example the fixed-degree case of Herstein's result on multiplicative commutators in division ring \cite{Her3} can be cast in this framework---and it is natural to ask to what extent the results given here can be extended to this more general framework.

We conclude this paper by raising a question.  Theorem \ref{thm:algorithm} is a theorem for ring identities, but one can instead fix a finitely generated commutative $\mathbb{Z}$-algebra $C$ (e.g., the ring of integers in a number field or a finite field) and work in the category of $C$-algebras and consider polynomial identities with coefficients in $C$.  It is possible that the approach we use could be used to deal with certain interesting classes of commutative base rings $C$, but we do not know of an algorithm that works for a general finitely generated commutative base ring $C$.  We note, however, that our approach applies to the case when $C$ a homomorphic image of $\mathbb{Z}$: in this case one can lift the identities to identities over the integers; the condition that our rings be $C$-algebras then puts an additional constraint on the characteristic of the ring when $C\neq\mathbb{Z}$.  In particular, we can use the algorithm provided in \S3, but where we restrict our focus to rings in the various classes whose characteristic divides the characteristic of $C$.  
\begin{question} Let $C$ be finitely generated commutative ring.  Given a finite presentation of $C$ and a finite set of polynomial identities $P_1=\cdots =P_m=0$, with $P_1,\ldots ,P_m\in C\{X_1,\ldots ,X_s\}$ for some $s\ge 1$, is there a decision procedure that takes the data from the presentation of $C$ and the polynomials $P_1,\ldots ,P_m$ as input and decides after a finite number of steps whether or not every $C$-algebra for which these identities all simultaneously hold is commutative? 
\end{question}
This question is especially interesting in the cases when $C$ is either a number ring (i.e., the ring of algebraic integers in a finite field extension of $\mathbb{Q}$) or when $C$ is a finite field.  In these cases, one might be able to extend the approach given in this paper to this setting, although it would require an extension of Theorem \ref{thm:trichotomy} to such $C$-algebras.
\vskip 2mm

\noindent{\bf Funding:} The work of Jason P. Bell was supported by NSERC Discovery Grant RGPIN-2016-03632. The work of Peter V. Danchev was partially supported by the Bulgarian National Science Fund under Grant KP-06 No 32/1 of December 07, 2019.

\end{document}